\documentclass[11pt]{article}
\usepackage{mathrsfs}
\usepackage[utf8]{inputenc}

\usepackage{amssymb}
\usepackage{enumitem}

\usepackage{amsthm}
\newtheorem{theorem}{Theorem}[section]
\newtheorem{lemma}[theorem]{Lemma}
\newtheorem{definition}[theorem]{Definition}
\newtheorem{example}[theorem]{Example}
\newtheorem{corollary}[theorem]{Corollary}
\newtheorem{proposition}[theorem]{Proposition}

\newcommand{\keywords}[1]{\par\medskip\noindent\textbf{Keywords:} #1}
\newcommand{\amsclass}[1]{\par\medskip\noindent\textbf{MSC 2020:} #1}
\usepackage[a4paper, left=2.5cm, right=2.5cm]{geometry}
\usepackage{lmodern}
\usepackage{setspace}
\usepackage{mathtools}
\usepackage{mathabx}
\newcommand{\dashmapsto}{\mapstochar\dashrightarrow}

\usepackage{tikz-cd}
\tikzset{%
    symbol/.style={%
        draw=none,
        every to/.append style={%
            edge node={node [sloped, allow upside down, auto=false]{$#1$}}}
    }
}
\tikzcdset{scale cd/.style={every label/.append style={scale=#1},
    cells={nodes={scale=#1}}}}

\tikzset{LA/.style = {
                      line width=#1, -{Straight Barb[length=3pt]}},
         LA/.default=1pt
        }

\tikzset{
  Earrow/.style={
    line width=0.4pt,
    double,
    -{Straight Barb[scale=1.0]},
    double distance=1.2pt
  }
}
  \tikzset{Tarrow/.style = {
                      line width=1.5pt, -{Straight Barb[length=3pt]}}
        }

\usepackage{multirow}

\newcommand{\Rel}{\mbox{\bf Rel}}

\newcommand{\RelMon}{\mbox{\bf RelMon}}
\newcommand{\RelFA}{\mbox{\bf RelFA}}
\newcommand{\Hom}{\mbox{\rm Hom}}
\newcommand{\EA}{\mbox{\bf EA}}

\newcommand{\PEA}{\mbox{\bf PEA}}

\newcommand{\POG}{\mbox{\bf POG}}
\newcommand{\Set}{\mbox{\bf Set}}
\newcommand{\eSSet}{\mathrm{Set}^\epsilon_\Delta}

\newcommand{\SSet}{\mathrm{Set}_\Delta}

\newcommand{\boxprod}{\mathrel{\Box}}
\newcommand{\Cell}{\mathrm{Cell}}
\newcommand{\Horn}{\mathrm{Horn}}
\newcommand{\relto}{\mathrel{\dashmapsto}}
\newcommand*\Sim{\includegraphics[scale=0.26]{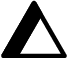}}
\newcommand{\eps}{{\bf \epsilon}}
\DeclareRobustCommand{\ELambda}{\text{\reflectbox{$\bf\Lambda$}}}
\def\dashmapsto{\mapstochar\dashrightarrow}

\newcommand{\markindex}[1]{\mathpalette\markindexhelper{#1}} 
\newcommand{\markindexhelper}[2]{%
  \fcolorbox{gray}{gray!20}{$\mathsurround=0pt#1#2$}%
}

\usepackage{scalerel}
\usepackage{hyperref}
\usepackage{cleveref}

\title{Simplicial Perspectives on (Pseudo) Effect Algebras}
\author{Dominik Lachman\thanks{Department of Algebra and Geometry,\\
Palacký University Olomouc,\\
17. Listopadu 12,\\
Czech Republic,\\
\texttt{dominik.lachman@upol.cz}}%
\space\thanks{Author acknowledge the support by the Czech Science Foundation (GAČR): project 24-14386L}} 
\date{}
\begin{document}

\maketitle  

\begin{abstract}
The study of Frobenius algebras in the category $\mathbf{Rel}$ via their nerve functor into simplicial sets has been introduced recently.
In this article, we focus on the particular case of effect algebras and pseudo effect algebras and investigate these using tools from combinatorial topology. 
To each (pseudo) effect algebra $E$, we associate a simplicial set with edge marking $N(E)$, called an $\epsilon$-simplicial set, 
and analyze its structural properties. In particular, we provide certain characterizations of effect algebras, orthoalgebras, and orthomodular posets 
among Frobenius algebras in $\mathbf{Rel}$. We show that the universal group $\mathrm{Gr}(E)$ of an effect algebra $E$ 
coincides with the first homology group of $N(E)$.  
For a pair of effect algebras $E,F$, we study the mapping space $[N(E),N(F)]$ and prove that the category of effect algebras 
can be enriched over the category of Frobenius algebras in $\mathbf{Rel}$. We extend this result to the category of pseudo effect algebras. Given pseudo effect algebras $E,F$, 
for the initial pseudo effect algebra $\underline{1}$, we show that the unique morphism $\underline{1}\to E$ induces a Kan fibration  
\(
[N(E),N(F)] \;\longrightarrow\; [N(\underline{1}),N(F)].
\)
We discuss how this result captures several structural features of conjugations 
in the theory of pseudo effect algebras.
\end{abstract}
\amsclass{03G12, 18N50, 18B10, 08A55}

\section*{Introduction}
Frobenius algebras---a classical concept in representation theory---are algebraic structures that admit both a monoid and a comonoid 
structure, tied together by the Frobenius identity. Following this definition, one may consider Frobenius algebras internal to any 
monoidal category $\mathcal{C}$. A prime example comes from the categorical approach to quantum information theory, where we model quantum processes 
in a dagger compact category $\mathcal{C}$ and internal Frobenius algebras play an essential role. 
Standard examples of dagger compact categories (the two running examples in~\cite{HV}) are
the category of finite-dimensional Hilbert spaces $\mathbf{Hilb}$ and the category of sets and relations $\mathbf{Rel}$. 

In this article, we concentrate on $\Rel$. It turns out that several well-studied structures can be organized as Frobenius 
monoids in $\Rel$: groupoids~\cite{CCH} (small non-empty categories of isomorphisms), effect algebras~\cite{PS2016} 
(fundamental structures in algebraic quantum logic), and coherent configurations~\cite{JJL}.  
These examples combine both algebraic and combinatorial aspects: the first two are partial monoids, while the latter lead to hypermonoids 
(multi-valued monoids). One of our main aims is to develop the combinatorial side of the coin. Our results show that certain essential 
properties of Frobenius algebras are inherently combinatorial. 
Our basic tool is the theory of simplicial sets $\SSet$, more concretely we will use so-called $\epsilon$-simplicial sets,
which are simplicial sets where some edges are marked. The corresponding category we denote as $\eSSet$. 
(Note that $\eSSet$ should not be confused with the category of marked simplicial sets $\SSet^+$, used in the study of Cartesian fibrations, even though the underlying idea is similar.)
The bridge between $\RelFA$ 
and $\eSSet$ is the fully faithful nerve functor
\begin{equation}\label{eq:eNerv}
    N\colon\RelFA \to \eSSet.
\end{equation} 
The study of groupoids (and more generally categories) from this perspective is classical~\cite{L}. The case of general $\RelFA$ was 
first considered in~\cite{MZ2020} (see also~\cite{CMS} and~\cite{OCS}). 
A slightly different approach is adopted in~\cite{My}. The main result in~\cite{My} provides a description of the image of the nerve 
functor~\eqref{eq:eNerv} in terms of right lifting properties. Hence, we can think of relational Frobenius algebras as certain fibrant 
objects in $\eSSet$. This result serves as our springboard for shifting from algebraic to combinatorial analysis. 
The result is analogous to the characterization of the image of groupoids as the simplicial sets admitting 
unique filling of each horn $\Lambda^n_i\subset\Delta^n$, $0\leq i\leq n$.

In this article, we will be mostly concerned with effect algebras, forming a category denoted $\EA$, and their non-commutative variant, pseudo effect algebras, forming 
a category denoted $\PEA$. However, many of our results are formulated for more general classes of $\RelFA$. 
Effect algebras are a central concept in algebraic quantum logic and were introduced by Foulis and Bennet~\cite{BF} in the 1990s, 
although essentially equivalent notions had appeared earlier (see~\cite{GG89}). Effect algebras are usually understood as algebraic 
structures and investigated by algebraic means; nevertheless, their analysis has traditionally been grounded in the theory of 
effect test spaces, developed by Foulis and Randall, which offer a combinatorial viewpoint.
In~\cite{F}, Foulis observed that classical test spaces associated with orthoalgebras (certain special class of effect algebras) are essentially simplicial complexes, 
and he encouraged further development of this connection with combinatorial topology. This paper can be seen as an elaboration of that 
program. However, instead of simplicial complexes, we use simplicial sets, and elements of an effect algebra $E$ correspond to edges of the associated simplicial set $N(E)$.

The rough motivation for the simplicial analysis of $\RelFA$ is that the category of simplicial sets has a very rich structure. 
Its theory is well developed and provides a wide range of tools. 
We demonstrate, through the following three main theorems, that this language is particularly natural for (pseudo) effect algebras.
\begin{itemize}
    \item In the theory effect algebras, there is an important construction of the so-called universal group, which is a functor $U\colon \EA\to \mathbf{Ab}$ to Abelian groups. 
    More precisely, the functor is taken to a certain category of po-groups. We show that the functor $U$ splits as $H_1\circ N$. 
    That is, $U(E)$ coincides with the first homology group of $N(E)$. While the observation is straightforward, 
    it appears to be new in the community around effect algebras.  
    \item Given $E,F\in\EA$, there is no obvious structure on $\mathrm{Hom}_{\EA}(E,F)$; it is merely a set. 
    In contrast, under passage to $\eSSet$, the hom-space $\mathrm{Map}(N(E),N(F))$ acquires a natural simplicial structure, 
    and it turns out that there exists $G\in\RelFA$ such that 
    $$N(G)\cong \mathrm{Map}(N(E),N(F)).$$ Consequently, we can arrange
    effect algebras as a category enriched over $\RelFA$.
    This is further generalized to the pseudo effect algebras.
    \item Given two pseudo effect algebras $E, F$, the vertices of the mapping space $$\mathrm{Map}(N(E),N(F))$$
    correspond to partial monoid morphisms (ignoring some additional structure), and edges correspond to conjugations between partial monoid morphisms. 
    It turns out that for a given partial monoid morphism $f\colon E\rightarrow F$, the data concerning conjugations of $f$ 
    are to a great extent determined by data about conjugations of the element $f(1)$. We make this observation precise 
    by proving that for pseudo effect algebras $E,F$ and the initial object 
$\underline{1}\in\PEA$, the restriction functor
\begin{equation}
    \mathrm{Map}(N(E),N(F)) \to \mathrm{Map}(N(\underline{1}),N(F))
\end{equation}
is a Kan fibration which also has the right lifting property with respect to $\partial\Delta^n\subset\Delta^n$, $n\geq 2$.  
\end{itemize}

The article is organized as follows.  
In Section~1 we recall the basic concepts: simplicial sets, enriched categories, Frobenius algebras, and the nerve functor.  
Section~2 gives the first taste of simplicial combinatorics in relation to our algebraic setting: we study the internal relation $\boxslash$ and provide a new characterization of orthoalgebras and orthomodular algebras.  
Section~3 shows that the well-known construction of the universal group of an effect algebra $E$ is essentially the first homology group of the associated simplicial set $N(E)$. 
Section~4 contains one of our main results: we study the mapping space $\mathrm{Map}(N(E),N(F))$ for effect algebras $E$ and $F$. 
We show that $\EA$ can be viewed as an $\RelFA$-enriched category.
Section~5 provides several technical lemmas needed in the sequel.  
Section~6 explores the mapping space $\mathrm{Map}(N(E),N(F))$ for pseudo effect algebras. We establish here the aforementioned main results. 

We note that some of the results in the Section~6 are presented strictly in the combinatorial language and are therefore considerably stronger than
the mentioned consequences for pseudo effect algebras. Similarly, just as results about $\infty$-groupoids are stronger than their implications for ordinary groupoids, our findings here have a similar flavor.

 We assume the reader is familiar with the fundamentals of category theory, 
but we do not presuppose advanced knowledge of simplicial sets, which is generally uncommon in the community of effect algebra 
researchers. Consequently, some arguments and proofs are presented in full detail, even when they might be considered routine or 
folklore by experts in category theory and simplicial sets.


\section{Preliminaries}
\subsection{Simplicial sets and $\epsilon$-simplicial sets}
Let us briefly recall the notion of simplicial sets. By $\Delta$ we denote the category whose objects are finite ordinals 
$[n] = \{0,1,\ldots,n\}$ and whose morphisms are order-preserving maps $f\colon [m]\to [n]$.
A functor (presheaf) $X\colon \Delta^{op}\to \Set$ is called a \emph{simplicial set}. 
All simplicial sets together with natural transformations form the category of simplicial sets, denoted by $\SSet$. 

For each object $[n]\in \Delta$, the representable simplicial set $\mathrm{Hom}_\Delta([n],-)$ is the standard $n$-simplex, 
denoted by $\Delta^n$. For $X\in\SSet$, we call the elements of $X_n = X([n])$ the \emph{$n$-simplices} of $X$. 
In particular, the $0$-simplices are called \emph{vertices}, and the $1$-simplices are called \emph{edges}. 
By the Yoneda lemma, there is a natural bijection between the set $X_n$ of $n$-simplices and the set of morphisms $\Delta^n \to X$. 
Following this bijection, we will sometimes identify these two sets.

Let $0\leq i\leq n$. We denote by $\delta^n_i\colon [n-1]\to [n]$ the unique injective morphism whose image does not contain $i$, 
and by $\sigma^n_i\colon [n+1]\to [n]$ the unique surjective morphism such that $i\in[n]$ has two preimages. 
For a simplicial set $X$, the map $d^n_i := X(\delta^n_i)\colon X_n \to X_{n-1}$ is called a \emph{face map}, 
and the map $s^n_i := X(\sigma^n_i)\colon X_n \to X_{n+1}$ is called a \emph{degeneracy map}. 
For an $n$-simplex $x\in X_n$, the language of face and degeneracy maps allows us to refer to its codimension-one faces 
$d^n_0(x),\ldots,d^n_n(x)$ and, similarly, to the $n+1$ degenerate $(n+1)$-simplices $s^n_0(x),\ldots,s^n_n(x)$ 
that naturally arise from $x$.

The Yoneda embedding $y\colon \Delta\to\SSet$ sends $\delta^n_i\colon [n-1]\to [n]$ to 
$y(\delta^n_i)\colon \Delta^{n-1}\hookrightarrow\Delta^n$, whose image we denote by $\partial_i\Delta^n$. 
Geometrically, $\partial_i\Delta^n$ is the face opposite the $i$-th vertex. 
Similarly, $y(\sigma^n_i)\colon \Delta^{n+1}\to\Delta^n$ corresponds geometrically to the retraction that collapses 
the edge between the $i$-th and $(i+1)$-th vertices. 
The Yoneda embedding gives us the following commutative diagram:
\begin{equation}
\begin{tikzcd}[row sep=25pt, column sep=40pt]
    \Delta^{n-1}\arrow[r,hook, "y(\delta^n_i)"]\arrow[rd,"d^n_i(x)"']&
    \Delta^n\arrow[d,"x"]&
    \Delta^{n+1}\arrow[ld,"s^n_i(x)"]\arrow[l, "y(\sigma^n_i)"']\\
    &X&
\end{tikzcd}
\end{equation}

Colimits and limits in $\SSet$ (as in any presheaf category) are computed objectwise. 
Similarly, for subsimplicial sets we will use the familiar set--theoretic operations $\cup$ and $\cap$, computed objectwise. 
We use the following standard notation for boundaries and horns: for all $0\leq i\leq n$,
\begin{equation}\label{eq:simpxAndHorn}
    \partial\Delta^n \;=\; \bigcup_{i=0}^n \partial_i\Delta^n,\quad\quad \Lambda^n_i      \;=\; \bigcup_{j=0,\, j\neq i}^n \partial_j\Delta^n.
\end{equation}
If the simplex $\Delta^n$, $n\geq 0$, is clear from the context, we will often address its faces as $\Delta^I$ for $I\subseteq\{0,\ldots,n\}$. 
Here $\Delta^I\subseteq \Delta^n$ denotes the face whose vertices belong to $I$. 
We also denote 
\begin{align*}
  \Cell         &= \{\,\partial\Delta^n \hookrightarrow \Delta^n \mid n \geq 0 \,\}, &
  \Horn         &= \{\,\Lambda_i^n \hookrightarrow \Delta^n \mid n \geq 1,\; 0\leq i \leq n \,\}, \\
  \Cell_{\geq k}&= \{\,\partial\Delta^n \hookrightarrow \Delta^n \mid n \geq k \,\}, &
  \Horn_{\geq k}&= \{\,\Lambda_i^n \hookrightarrow \Delta^n \mid n \geq k,\; 0\leq i \leq n \,\}.
\end{align*}
We will also adopt the following relaxed notation: if $\sigma$ is a simplex isomorphic to $\Delta^n$ and $x$ is a vertex of $\sigma$, 
then we denote the face of $\sigma$ opposite to $x$ by $\partial_x\sigma$.

For our purposes we will need slightly richer structures than simplicial sets. 
In particular, we want to distinguish certain edges by marking. 
To achieve this, we consider the category $\Sim$ freely generated by $\Delta$ and an arrow $[1]\to [\epsilon]$, 
where $[\epsilon]$ is an additional object. A presheaf on $\Sim$ is called an \emph{$\epsilon$-simplicial set}. 
By $\eSSet$ we denote the category of all $\epsilon$-simplicial sets together with natural transformations. 
Note that an $\epsilon$-simplicial set $X$ is essentially a simplicial set together with a map $X(\epsilon)\to X(1)$. 
We call the edges in the image of this map \emph{$\epsilon$-edges}, and the elements of $X(\epsilon)$ we call \emph{witnesses}.

By $\Sim^n$ we denote the standard $n$-simplex $\Delta^n$ where the edge $\Delta^{\{0,n\}}$ between the first and the last vertex is an $\epsilon$-edge.
For each $0\leq i\leq n$, the $\epsilon$-simplicial set $\partial_i\Sim^n$ equals $\partial_i\Delta^n$ with the edge $\Delta^{\{0,n\}}$ being an $\epsilon$-edge, 
whenever it belongs to $\partial_i\Delta^n$ (i.e., $i\neq 0,n$).
Following the same pattern, we refer to the faces of $\Sim^n$, $n\geq 0$, using the notation $\Sim^I$, $I\subseteq\{0,\ldots,n\}$. For our purposes, the following simplicial sets are important
\begin{equation}
    \ELambda^n_0=\bigcup_{i=1}^n \partial_i\Sim^n,\quad\quad \ELambda^n_n=\bigcup_{i=0}^{n-1} \partial_i\Sim^n,
\end{equation} 
we call them \emph{$\epsilon$-horns}.

Working with diagrams, we will mark the $\epsilon$-edges by thick arrows. For example, the following diagram are $\Sim^n$, for $n=1,2,3$:
\[
\begin{tikzcd}
  0\arrow[r,Tarrow]&1  
\end{tikzcd}\quad\quad
\begin{tikzcd}[row sep=20pt, column sep=15pt]
    &2&\\
    0\arrow[rr]\arrow[ru,Tarrow]&&1\arrow[lu]
\end{tikzcd}
\quad\quad
    \begin{tikzcd}[row sep=10pt, column sep=20pt]
  & 3 & \\
  &[20pt] & \\
  & 2 \arrow[uu] & \\
  0 
    \arrow[rr, bend right=5] 
    \arrow[ruuu, Tarrow, bend left=5] 
    \arrow[ru] 
  & & 
  1 
    \arrow[lu] 
    \arrow[luuu, bend right=5]
\end{tikzcd}
\]

The product of two standard simplices $\Delta^n\times\Delta^m$ is the simplicial set whose $k$-simplices correspond to poset morphisms $[k]\to [m]\times [n]$. 
Consequently, $\Delta^n\times\Delta^m$ has dimension $m+n$, and the number of its top-dimensional simplices is \(\binom{m+n}{m}\).

\subsection{Lifting properties}
Let $f\colon A \to B$ and $g\colon X \to Y$ be morphisms in a category $\mathcal{C}$. 
We write
\(
f \,\boxslash\, g
\)
if, for every commutative square
\begin{equation}\label{eq:liftProb}
\begin{tikzcd}
A \arrow[r] \arrow[d,"f"'] & X \arrow[d,"g"] \\
B \arrow[r] \arrow[ru,dotted,"l"] & Y
\end{tikzcd}
\end{equation}
there exists a lift $l\colon B \to X$. 
In this case, $f$ is said to have the \emph{left lifting property} (LLP) with respect to $g$, 
and $g$ the \emph{right lifting property} (RLP) with respect to $f$. 
If the lift in~\eqref{eq:liftProb} is unique, we say that $f$ (resp.~$g$) has the \emph{unique} LLP (resp.~RLP).  
If $Y$ is the terminal object and $g$ the unique morphism $X \to Y$, we write $f \,\boxslash\, X$. 
If, moreover, $f$ is a monomorphism, we sometimes say that $X$ \emph{admits a filling of} $A$.

\begin{example}
A morphism of simplicial sets $f\colon X \to Y$ is a Kan fibration if $\mathrm{Horn}\boxslash f$.  
If $f$ has the unique RLP with respect to $\mathrm{Horn}$, we call it a minimal Kan fibration.
\end{example}

For a class of morphisms $\mathcal{K}$ we denote by
\[
\overline{\mathcal{K}} \;=\; {}^\boxslash(\mathcal{K}^\boxslash)
\]
the closure of $\mathcal{K}$ with respect to the Galois connection induced by $\boxslash$.
The class $\overline{\mathcal{K}}$ is closed under pushouts, retracts, and (transfinite) compositions; 
a class of morphisms with this property is called \emph{weakly saturated}. 
For many reasonable categories (such as $\SSet$ and $\eSSet$) the closure operator with respect to these operations
coincides with~$\overline{(-)}$.

For $X\in\SSet$ and integer $n\geq 0$, we denote by $\mathrm{sk}_n X$ the simplicial set obtained from $X$ by omitting all simplices of dimension greater than $n$. 
Hence, for $l\geq n$, each $l$-simplex of $\mathrm{sk}_nX$ is degenerate. 

The following lemma is well known.

\begin{lemma}\label{lem:fillingCells}
    Let $X,Y\in \eSSet$ and $n\geq 0$. If $\mathrm{sk}_n Y\subseteq X\subseteq Y$, then $X\hookrightarrow Y \in \overline{\Cell_{\geq n+1}}$. 
\end{lemma}

We say that a simplicial set is \emph{reduced} if it has a single vertex. 
Given $X\in\SSet$, its reduction is the simplicial set $X_{\mathrm{red}}$ obtained by identifying all of its vertices:
\begin{equation}\label{eq:reduction}
    X_{\mathrm{red}} \;=\; X\coprod_{\mathrm{sk}_0X}\Delta^0.
\end{equation}
When verifying that a monomorphism of reduced simplicial sets belongs to a certain weakly saturated class, 
the following lemma allows us to pass to the non-reduced case.

\begin{lemma}\label{lem:reduction}
  Let $X,Y\in\SSet$ with $\mathrm{sk}_0 Y\subseteq X\subseteq Y$. 
  Then the induced map $X_{\mathrm{red}}\to Y_{\mathrm{red}}$ is a pushout of $X\hookrightarrow Y$.
\end{lemma}
\begin{proof}
   Consider the commutative diagram:
\begin{equation}
    \begin{tikzcd}
        \mathrm{sk}_0 Y\arrow[d]\arrow[r] & X\arrow[d]\arrow[r,hook] & Y\arrow[d]\\
        \Delta^0\arrow[r] & X_{\mathrm{red}}\arrow[r] & Y_{\mathrm{red}}
    \end{tikzcd}
\end{equation}
By definition of reduction, the left-hand square and the outer rectangle are pushouts. 
Consequently, the right-hand square is also a pushout.
\end{proof}

 
 \subsection{Categorical background}

One of our main results shows that the category of effect algebras can be viewed as a category enriched over the category of Frobenius algebras with cartesian monoidal structure. For this reason, we briefly recall the concept of enriched categories. A standard reference on enriched category theory is~\cite{Kelly2005}.

A \emph{monoidal category} is a category $\mathcal{C}$ equipped with a bifunctor $\otimes\colon \mathcal{C}\times\mathcal{C}\to\mathcal{C}$, an object $I\in\mathcal{C}$, and, for each $A,B,C\in \mathcal{C}$, natural isomorphisms 
\begin{gather}
\alpha_{A,B,C}\colon (A \otimes B)\otimes C \;\cong\; A\otimes (B\otimes C),\\
\lambda_A\colon I\otimes A \;\cong\; A,\quad \rho_A\colon A\otimes I \;\cong\; A,
\end{gather}
which satisfy the usual coherence conditions, i.e. the \emph{pentagon and triangle identities} (see~\cite{Kelly2005}).  
Any category with finite products admits a \emph{cartesian monoidal structure}, where $\otimes$ is the product and $I$ the terminal object. The natural isomorphisms 
$\alpha$, $\lambda$, and $\rho$ are in this case induced by the universal property of products and the terminal object.

Let $(\mathcal{V},\otimes, I)$ be a monoidal category. A \emph{$\mathcal{V}$-category} $\mathcal{C}$ consists of the following data: 
\begin{itemize}
    \item a class of objects $\mathrm{Obj}(\mathcal{C})$;
    \item for each pair of objects $A,B\in\mathrm{Obj}(\mathcal{C})$, an object $\mathcal{C}(A,B)\in\mathcal{V}$, called the \emph{hom-object};
    \item for each triple $A,B,C\in\mathrm{Obj}(\mathcal{C})$, a morphism 
    \[
    \mathcal{C}(B,C)\otimes\mathcal{C}(A,B)\longrightarrow \mathcal{C}(A,C)
    \]
    in $\mathcal{V}$, called the \emph{composition morphism};
    \item for each $A\in\mathrm{Obj}(\mathcal{C})$, a morphism $I\to \mathcal{C}(A,A)$ in $\mathcal{V}$, called the \emph{identity element}.
\end{itemize}
The composition and identity morphisms are subject to coherence conditions analogous to the associativity and identity axioms in ordinary categories. 

Given a $\mathcal{V}$-category $\mathcal{C}$, there is a standard construction of an \emph{underlying} ordinary category $\mathcal{C}_0$. The objects of $\mathcal{C}_0$ are those of $\mathcal{C}$, and for $A,B\in \mathrm{Obj}(\mathcal{C})$ the hom-set is defined by
\[
\mathrm{Hom}_{\mathcal{C}_0}(A,B) \;=\; \mathrm{Hom}_\mathcal{V}(I,\mathcal{C}(A,B)).
\]
The composition of $f\colon I\to \mathcal{C}(A,B)$ and $g\colon I\to \mathcal{C}(B,C)$ is given by
\begin{equation}
    I\xrightarrow{\cong} I\otimes I
    \xrightarrow{g\otimes f} \mathcal{C}(B,C)\otimes \mathcal{C}(A,B)
    \xrightarrow{\circ} \mathcal{C}(A,C).
\end{equation} 



\subsection{Algebraic structures}
\begin{definition}[\cite{BF}]
    An \emph{effect algebra} is a structure $(E,\oplus,0,1)$, where
    \begin{enumerate}[label=(\arabic*)]
        \item $(E,\oplus,0)$ is a partial commutative monoid;
        \item every $a \in E$ has a unique \emph{orthosupplement} $a'$ such that $a \oplus a' = 1$;\label{item:EAorthosupplement}
        \item for each $a \in E$, if $a \oplus 1$ is defined, then $a=0$. 
    \end{enumerate}
\end{definition}
Every effect algebra $E$ carries a canonical partial order defined by 
\begin{equation}
    a \leq b \quad \text{if and only if} \quad \exists c \in E \text{ such that } b = a \oplus c.
\end{equation}
With respect to this order, $0$ is the least element and $1$ is the greatest element.
For two elements $a,b \in E$, we write $a \perp b$ to indicate that $a \oplus b$ is defined, and say that $a$ and $b$ are \emph{orthogonal}. 
An important consequence of the uniqueness condition in~\ref{item:EAorthosupplement} is the cancellation property:
\begin{equation}
    a \oplus b_1 = a \oplus b_2 \;\;\implies\;\; b_1 = b_2.
\end{equation}
Among effect algebras, two subclasses are of particular interest:
An effect algebra is called \emph{orthomodular} if for all $a,b\in E$, such that $a\perp b$, we have $a\oplus b$ equals the supremum $a\vee b$. Moreover, an effect algebra is called \emph{orthoalgebra}, if $a\perp a \implies a=0$. 
It is worth noting that every orthomodular effect algebra is also an orthoalgebra. 
For more details on effect algebras, see~\cite{DvPu}.

It is proved in~\cite{PS2016} that every effect algebra is a Frobenius algebra in the category of sets and relations, denoted $\Rel$. 
The category $\Rel$ has a cartesian monoidal structure.
A Frobenius algebra is an algebra which is both a monoid and a comonoid, where these two structures are connected by the Frobenius identity.
We begin by giving the definition of a monoid in $\Rel$, formulated in elementary terms. Note that we use dashed arrows for morphisms in $\Rel$ to stress the fact that they are relations, not mappings.

\begin{definition}
    A \emph{monoid in $\Rel$} is a triple $(A,\mu,\eta)$, where $A$ is a set, 
    $\mu \colon A \times A \dashrightarrow A$ is a relation, and $\eta \subseteq A$, such that for all $a,b,c,d \in A$ the following conditions hold:
    \begin{enumerate}
        \item there exist $s,t \in \eta$ such that $\mu\colon (a,s) \dashmapsto a$ and $\mu\colon (t,b) \dashmapsto b$;
        \item for each $r \in \eta$, if $\mu\colon (r,a) \dashmapsto b$ or $\mu\colon (a,r) \dashmapsto b$, then $a=b$;
        \item there exists $x \in A$ such that $\mu\colon (a,b) \dashmapsto x$ and $\mu\colon (x,c) \dashmapsto d$ 
        if and only if there exists $y \in A$ such that $\mu\colon (a,y) \dashmapsto d$ and $\mu\colon (b,c) \dashmapsto y$.
    \end{enumerate}
\end{definition}

\begin{example}\label{catInMon}
    Every small category is a monoid in $\Rel$, where $\eta$ consists of all identities and $\mu$ is the composition relation.
\end{example}
Note that in $\Rel$, choosing a subset $\eta \subseteq A$ is equivalent to choosing a morphism $\eta\colon \bullet \to A$. 
Under the identification $\bullet \times A = A = A \times \bullet$, the conditions (1)–(2) translate to the equations
\[
    \mu \circ (\eta \times \mathrm{id}) = \mathrm{id} = \mu \circ (\mathrm{id} \times \eta).
\]
Similarly, under the identification $(A \times A) \times A \cong A \times (A \times A)$, the associativity condition (3) translates to
\[
    \mu \circ (\mu \times \mathrm{id}) = \mu \circ (\mathrm{id} \times \mu).
\]

Given two monoids in $\Rel$, $(A,\mu_A,\eta_A)$ and $(B,\mu_B,\eta_B)$, 
a \emph{morphism of monoids} $$f\colon (A,\mu_A,\eta_A) \to (B,\mu_B,\eta_B)$$ is a mapping $f\colon A \to B$ such that $f(\eta_A) \subseteq \eta_B$ and whenever $\mu_A\colon (a,b) \dashmapsto c$, we also have $\mu_B\colon (f(a),f(b)) \dashmapsto f(c)$. 
With this notion of morphism, monoids in $\Rel$ form a category, which we denote by $\RelMon$.

A \emph{comonoid in $\Rel$} is a triple $(A,\delta,\epsilon)$ such that $(A,\delta^{-1},\epsilon)$ is a monoid in $\Rel$. 
For a comonoid, we may identify $\epsilon$ with a morphism $\epsilon\colon A \to \bullet$. 
In $\Rel$, every monoid yields a comonoid and vice versa by inverting the relations. 
This duality holds in every monoidal dagger category (see~\cite{HV}).
\begin{definition}
    A \emph{Frobenius algebra in $\Rel$} is a quintuple $(A,\mu,\eta,\delta,\epsilon)$, 
    where $A$ is a set, $\mu\colon A \times A \dashrightarrow A$ and $\delta\colon A \dashrightarrow A \times A$ are relations, and $\eta,\epsilon \subseteq A$, 
    such that:
    \begin{enumerate}
        \item $(A,\mu,\eta)$ is a monoid in $\Rel$;
        \item $(A,\delta,\epsilon)$ is a comonoid in $\Rel$;
        \item for all $a,b,c,d \in A$, there exists $x \in A$ such that $\mu\colon (a,x) \dashmapsto c$ and $\delta\colon b \dashmapsto (x,d)$ 
        if and only if there exists $y \in A$ such that $\delta\colon a \dashmapsto (c,y)$ and $\mu\colon (y,b) \dashmapsto d$.
    \end{enumerate}
\end{definition}
Several important structures can be organized as Frobenius algebras in $\Rel$.
\begin{example}
       A groupoid $G$ gives rise to a Frobenius algebra,
 where $\mu\colon (f,g)\dashmapsto h\Leftrightarrow f\circ g=h\Leftrightarrow \delta\colon h\dashmapsto (f,g)$ and $\eta=\epsilon$ consist of all identities on objects.
\end{example}
\begin{example}[\cite{PS2016}]
    An effect algebra $(E,\oplus,0,1)$ gives rise to a Frobenius algebra with the following data:
    \begin{align*}
        \mu\colon (a,b) \dashmapsto c &\;\;\Leftrightarrow\;\; a \oplus b = c,\\
        \delta\colon c \dashmapsto (a,b) &\;\;\Leftrightarrow\;\; c' = a' \oplus b' \;\;\Leftrightarrow\;\; c = a \odot b,\\
        \eta &= \{0\}, \quad \epsilon = \{1\}.
    \end{align*}
\end{example}

Note that in the aforementioned examples the comultiplication is derivable from multiplication and the counit. 
This is a general property of Frobenius algebras (not only in $\Rel$).
For this reason, it is natural to consider the following auxiliary structures: 
\begin{definition}
    An $\epsilon$-monoid in $\Rel$ is a monoid $(A,\mu,\eta)$ together with a subset $\epsilon \subseteq A$ (equivalently, a morphism $\epsilon \colon A \to \bullet$). 
\end{definition}
A morphism of $\epsilon$-monoids $f \colon (A,\mu_A,\eta_A,\epsilon_A) \to (B,\mu_B,\eta_B,\epsilon_B)$ 
is a morphism of relational monoids such that $f(\epsilon_A) \subseteq \epsilon_B$. 
We denote the category of $\epsilon$-monoids in $\Rel$ by $\epsilon\text{-}\RelMon$.

The reason for introducing $\epsilon$-monoids is that the category $\epsilon\text{-}\RelMon$ is the natural domain for a fully faithful nerve functor 
\begin{equation}\label{eq:nerve}
    N \colon \epsilon\text{-}\RelMon \longrightarrow \eSSet.
\end{equation}
In order to establish~\eqref{eq:nerve}, let us first define an embedding 
\begin{equation}
    i \colon \Sim \hookrightarrow \epsilon\text{-}\RelMon.
\end{equation}
Each $[n]$ is a poset, hence a category, and thus gives a relational monoid $i([n])$ (see Example~\ref{catInMon}).
For $i([\epsilon])$, we take the same monoid as for $[1]$, except that the unique non-unit element is in $\epsilon$.
It is straightforward to verify that this extends to a functor $i \colon \Sim \to \epsilon\text{-}\RelMon$. 
The nerve functor $N$ from~\eqref{eq:nerve} is then defined by
\[
    N(A)(-) = \mathrm{Hom}(i(-),A).
\]
It was proved in~\cite{My} that $N$ is a fully faithful embedding. 
Moreover, $N$ preserves finite products: for a pair of $\epsilon$-monoids $A$, $B$ and an object $[x] \in \Sim$, we have
\begin{gather}\label{eq:nerveProduct}
    N(A \times B)([x]) = \mathrm{Hom}(i([x]),A \times B) 
    = \mathrm{Hom}(i([x]),A) \times \mathrm{Hom}(i([x]),B) \notag\\
    = N(A)([x]) \times N(B)([x]).  
\end{gather}
Similarly, the terminal object in $\RelFA$ is a one-element Frobenius algebra denoted $\underline{0}$ and 
\begin{equation}\label{eq:nervTer}
    N(\underline{0}) \cong \Sim^0,
\end{equation}
which is the terminal object in $\SSet^\epsilon$. We can summarize the aforementioned thoughts as follows:
\begin{proposition}\label{prop:monoidalFunctor}
    The nerve functor $N\colon\RelFA\to\SSet^\epsilon$ is a fully faithful monoidal embedding of cartesian monoidal categories.
\end{proposition}

\begin{example}
    For an effect algebra $E$, each morphism $i([n]) \to E$ can be identified with an orthosum $a = a_1 \oplus \cdots \oplus a_n$ with $n$ summands. 
    Hence, $N(E)$ contains a unique vertex ($0=0$), edges are in bijection with elements of $E$, and $2$-simplices correspond to pairs $(a,b) \in E^2$ such that $a \oplus b$ is defined.
\end{example}
It was shown in~\cite{My} that for a relational monoid $A$ the following equivalences hold:
\begin{itemize}
    \item $A$ satisfies the left cancellation property $\iff N(A)$ admits fillings of $\Lambda^3_0$;
    \item $A$ is a partial monoid $\iff N(A)$ admits fillings of $\Lambda^3_1$ $\iff N(A)$ admits fillings of $\Lambda^3_2$;
    \item $A$ satisfies the right cancellation property $\iff N(A)$ admits fillings of $\Lambda^3_3$.
\end{itemize}
In particular, for an effect algebra $E$, the $\epsilon$-simplicial set $N(E)$ admits fillings of all $3$-horns 
\begin{equation}\label{eq:3horns}
    \Lambda^3_i \subset \Delta^3,\quad i=0,1,2,3.
\end{equation}

It is a property of Frobenius algebras in $\Rel$ that all the aforementioned properties are equivalent. 
Moreover, we have the following characterisation of the image of
$N \colon \RelFA \to \eSSet$ (see \cite{My}, Theorem~18 and Lemma~15):

\begin{theorem}\label{thm:char}
    An $\epsilon$-simplicial set $X$ is of the form $X \cong N(F)$ for some Frobenius algebra $F$ in $\Rel$ if and only if $X$ has the unique RLP with respect to
    \begin{enumerate}
        \item[(i)] $\ELambda^n_0 \subset \Sim^n$ and $\ELambda^n_n \subset \Sim^n$ for $n \geq 1$ (\emph{$\epsilon$-horns}),
        \item[(ii)] $\partial\Delta^n \subset \Delta^n$ for $n \geq 3$ (\emph{$2$-coskeletality}),
    \end{enumerate}
    and $X$ has the RLP with respect to
    \begin{enumerate}
        \item[(iii)] $\partial_0 \Delta^3 \cup \partial_2 \Delta^3 \subset \Delta^3$ (\emph{associativity}). 
    \end{enumerate}
\end{theorem} 
Let $F \in \RelFA$. We adopt the following convention.  
Since the elements of $F$ are in bijection with the edges of $N(F)$, we will occasionally identify these two sets. Moreover, for each $a \in F$, we denote by $s(a)$ and $t(a)$ the source and target of the edge $a$ in $N(F)$, respectively. Two elements $a,b \in F$ are \emph{composable}---that is, there exists some $c$ with 
\(
  \mu\colon (b,a) \dashmapsto c,
\)
if and only if there is a $2$-simplex $\sigma \in N(F)$ such that
\[
  d_0(\sigma) = a, \quad d_1(\sigma) = c, \quad d_2(\sigma) = b.
\]
In this case we write $a \perp b$.

For $F \in \RelFA$, we now describe consequences of the property 
\[
  (\ELambda^n_i \subset \Sim^n) \boxslash N(F), \quad n \geq 1, i\in\{0,n\},
\]
from Theorem~\ref{thm:char}, that $N(F)$ admits fillings of $\epsilon$-horns.
\begin{itemize}
  \item For $n=1$, it follows that each vertex of $N(F)$ is both the source of some $\epsilon$-edge and the target of some $\epsilon$-edge.
  \item For $n=2$, it follows that for each pair of edges $a,e$ with $e \in \epsilon$, we have
\begin{enumerate}[label=(\roman*)]
  \item If $s(a) = s(e)$, then there exists a third edge $\beta(a)$ such that 
  \[
    \mu\colon (\beta(a),a) \dashmapsto e.
  \]
  \item If $t(a) = t(e)$, then there exists a third edge $\alpha(a)$ such that 
  \[
    \mu\colon (a,\alpha(a)) \dashmapsto e.
  \]
\end{enumerate}
\end{itemize}
\section{Orthogonality relation internal to Frobenius algebras}
One may think of a relational monoid as ``almost'' a category, where elements correspond to arrows. 
In this sense, we can investigate the relation of having the \emph{right lifting property}, denoted by $\boxslash$. 
It turns out that this relation is very meaningful for effect algebras. 
We first give the definition in terms of Frobenius algebras:
\begin{definition}\label{def:internalBoxslash}
Let $F$ be a Frobenius algebra and let $a,b \in F$. 
    We write $a \boxslash b$ if for each $p,q,d$ such that $\mu\colon (q,a)\relto d$ and $\mu\colon (b,p)\relto d$, there exists the following $3$-simplex in $N(F)$:
\begin{equation}\label{diag:ortho}
    \begin{tikzcd}[row sep=10pt, column sep=20pt]
  & 3 & \\
  &[20pt] & \\
  & 2 \arrow[uu, "b"{description}] & \\
  0 
    \arrow[rr, "a"{description}, bend right=5] 
    \arrow[ruuu, "d"{description}, bend left=5] 
    \arrow[ru, "p"{description}] 
  & & 
  1 
    \arrow[lu, "l"{description}] 
    \arrow[luuu, "q"{description}, bend right=5]
\end{tikzcd}
\end{equation}
\end{definition}

\begin{proposition}\label{prop:itoiii}
    For an element $a$ of a Frobenius algebra $(F,\mu,\delta,\eta,\epsilon)$ the following are equivalent:
    \begin{enumerate}[label=(\roman*)]
        \item $a$ has a right inverse $c$ (that is, there exists $r\in\eta$ with $\mu\colon (a,c)\dashmapsto r$);
        \item for each $b\in F$, we have $b \perp a$  whenever $s(b) = t(a)$;
        \item there exists $e\in\epsilon$, with $e \perp a$.
    \end{enumerate}
    Moreover, (i)--(iii) hold whenever $\epsilon \boxslash a$.
\end{proposition}
\begin{proof}
(i)$\implies$(ii): Suppose there exists $c\in F$ and $r\in\eta$ such that $\mu\colon (a,c)\dashmapsto r$. 
For any $b$ with $s(b) = t(a)$, we obtain (ii) by associativity rule applied to $\mu\colon(b,r)\dashmapsto b$ and $\mu\colon(a,c)\dashmapsto r$.

(ii)$\implies$(iii): Follows immediately, since there is $e\in\epsilon$ with $s(e)=s(a)$.

(iii)$\implies$(i): We show that there exists the following simplex $\sigma \in N(F)_3$:
\[
    \begin{tikzcd}[row sep=12pt, column sep=24pt]
  & 3 & \\
  &[20pt] & \\
  & 2 \arrow[uu, Tarrow,"e"{description}] & \\
  0 
    \arrow[rr, "\alpha(b)"{description}, bend right=5] 
    \arrow[ruuu, Tarrow, "e"{description}, bend left=5] 
    \arrow[ru, "s(e)"{description}] 
  & & 
  1 
    \arrow[lu, "a"{description}] 
    \arrow[luuu, "b"{description}, bend right=5]
\end{tikzcd}
\]
By the premise, the face $d_0\sigma$ exists.
The faces $d_1\sigma$ and $d_2\sigma$ also exist, and together they form an $\epsilon$-horn $\ELambda^3_3\to N(F)$ 
which admits a filling --- the desired simplex $\sigma$. The resulting face $d_3\sigma$ witnesses that $\alpha(b)$ is the required right inverse.

Finally, to see that $\epsilon \boxslash a$ implies (i), consider the lifting problem:
\[
\begin{tikzcd}[row sep=large, column sep=large]
0
  \arrow[d, Tarrow, "e" description]
  \arrow[r, "\alpha(a)" description]
  \arrow[rd, Tarrow, "e" description] & 
2 \arrow[d, "a" description] \\
1 \arrow[r, "t(e)" description] & 3
\end{tikzcd}
\]
The solution gives us the desired right inverse.
\end{proof}
        \begin{theorem}\label{thm:iv}
            Let $(F,\mu,\delta,\eta,\epsilon)$ be a cancellative Frobenius algebra and $a$ its element. Then the equivalent properties (i--iii) from Proposition~\ref{prop:itoiii} hold iff
            \begin{enumerate}
                \item[(iv)] \(F\boxslash a\), iff
                \item[(v)] \(\epsilon\boxslash a\).
            \end{enumerate}
        \end{theorem}
\begin{proof}
    
    We prove (i)\(\implies\)(iv). Let \(c\) be a right inverse of \(a\). Given the following situation, where $b\in F$ is any element:
    \[
\begin{tikzcd}[row sep=7pt, column sep=30pt]
0 
  \arrow[dd, "b" description] 
  \arrow[r, "p" description] 
  \arrow[rdd, "d" description] & 3 
  \arrow[dd, "a" description]  & \\
  &&2\arrow[lu,"c"]\arrow[ld,"r"]\\
1 
  \arrow[r, "q" description] & 
4& 
\end{tikzcd}
\]
    We need to prove there is a filling of the whole $3$-simplex $\Delta^{\{0,1,3,4\}}\to N(F)$. Consider the diagram with two more edges as follows:
    
    \[
\begin{tikzcd}[row sep=7pt, column sep=30pt]
0 \arrow[rrd,bend left = 50pt,"d"description]
  \arrow[dd, "b" description] 
  \arrow[r, "p" description] 
  \arrow[rdd, "d" description] & 3 
  \arrow[dd, "a" description]  & \\
  &&2\arrow[lu,"c"]\arrow[ld,"r"]\\
1 \arrow[rru, bend right= 50pt, "q"description]
  \arrow[r, "q" description] & 
4& 
\end{tikzcd}
\]
We will first fill the simplex $\Delta^{\{0,2,3,4\}}\to N(F)$. We already have all its faces until the image of $\Delta^{\{0,2,3\}}$, hence we can apply the cancellation property to fill the corresponding $3$-horn. So obtained $2$-simplex $\Delta^{\{0,2,3\}}\to N(F)$ together with $\Delta^{\{0,1,2\}}\to N(F)$ witnessing $\mu\colon (q,b)\dashmapsto d$ form two faces of $\Delta^{\{0,1,2,3\}}\to N(F)$. Applying associativity, we obtain the whole $3$-simplex and in particular its face $\Delta^{\{0,1,3\}}\to N(F)$, which is the third face of the $3$-simplex $\Delta^{\{0,1,3,4\}}\to N(F)$ in question. We finish by filling the corresponding $3$-horn.

    (iv)\(\implies\) (v) is trivial and (v)\(\implies\) (i) is already established in Proposition~\ref{prop:itoiii}.
\end{proof}

    

    
\begin{theorem}
    A Frobenius algebra \(E\) with simplicial set $N(E)$ is an effect algebra if and only if it satisfies 
    \begin{enumerate}
        \item[(i)] commutativity,
        \item[(ii)] cancellation property (equivalently partiality),
        \item[(iii)] $\epsilon^{\boxslash}=\eta$,
        \item[(iv)] $\eta$ is a singleton.
    \end{enumerate}
    Moreover, the conjunction (iii) $\wedge$ (iv) is equivalent to
    \begin{enumerate}
        \item[(v)] for each $e\in\epsilon$, $e^\boxslash$ is a singleton.
    \end{enumerate}
\end{theorem}
\begin{proof}
    Assume first \(E\) is an effect algebra. Statements (i--ii) and (iv) obviously hold. By Proposition~\ref{prop:itoiii}, each element $a\in E$ satisfying $\epsilon\boxslash a$ has a right inverse. Since \(E\) is an effect algebra, such $a$ need to be \(0\).

    Next assume a Frobenius algebra satisfying (i--iv). It induces a cancellative partial commutative monoid $(E,\oplus,0)$, where $0$ is the unique element in $\eta$. In each Frobenius algebra, counit elements are in bijection (induced by the rotation \(\hat{\alpha}\)) with unit elements. Hence, (iv) implies $\epsilon$ has a unique element; denote it $1$.
    
    Orthosupplements are given by the rotation \(\hat{\alpha}\). It remains to verify that for each $a\in E$, if \(1\oplus a\) is defined, then $a=0$. But this is immediate from Theorem~\ref{thm:iv} and the assumption (iii).

    Finally, let us verify (iii)\(\wedge\)(iv) iff (v). Assume (v) holds. From \(\epsilon\boxslash\eta\) we deduce at once (iii) and (iv). In order to show the opposite implication, assume $e\in \epsilon$ and some $a\in E$ such that $e\boxslash a$. It is enough to show $a\in\eta$. By (iv), $\epsilon$ is a singleton, hence equal to $\{e\}$. Consequently, $\epsilon\boxslash a$, and (iii) entails $a\in \eta$.  
\end{proof}
Let us reveal the meaning of \(\boxslash\) in the case of effect algebras. 
\begin{lemma}\label{lem:charBoxSlash}
    Assume an effect algebra \(E\) and two elements \(a,b\in E\). We have \(a\boxslash b\) iff \(a\perp b\) and \(a\oplus b=a\vee b\).
\end{lemma}
\begin{proof}
    Assume \(a\boxslash b\) and consider the following situation:
    \begin{equation}\label{diag:orthoEA1}
        \begin{tikzcd}
0\arrow[d,"a" {description}]\arrow[r,"b'" {description}]\arrow[rd,Tarrow,"1" {description}]&2\arrow[d,"b" {description}]\\  
1\arrow[r,"a'" {description}]&3       
        \end{tikzcd}
    \end{equation}
    For \(d=1\), a solution gives \(l\), such that \(a\oplus l\oplus b=1\), which gives \(a\perp b\).  In general situation, the following
     data, essentially encodes just upper bound \(a,b\leq d\):
\begin{equation}\label{diag:orthoEA2}
        \begin{tikzcd}
0\arrow[d,"a" {description}]\arrow[r]\arrow[rd,"d" {description}]&2\arrow[d,"b" {description}]\\  
1\arrow[r]&3       
        \end{tikzcd}
    \end{equation}
The presence of solution to \eqref{diag:orthoEA2} gives us \(l\) such that  \(a\oplus b\oplus l=d\). That is \(a\oplus b\leq d\).

    Conversely, assume \(a\perp b\) and \(a\oplus b=a\vee b\). Then for any lifting problem~\eqref{diag:orthoEA2} we can set \(l=d\ominus (a\oplus b)\). For such \(l\) we have
    \begin{align*}
        l\oplus a\;=\;& (d\ominus (a\oplus b))\oplus a\;=\;d\ominus b,\\
        l\oplus b\;=\;& (d\ominus (a\oplus b))\oplus b\;=\;d\ominus a.\\
    \end{align*}
\end{proof}
\begin{theorem}
    An effect algebra is 
    \begin{enumerate}
        \item orthomodular if and only if \(\perp\,=\,\boxslash\);
        \item orthoalgebra if and only if \(\alpha\subseteq \boxslash\).    
    \end{enumerate}
\end{theorem}
\begin{proof}
        According to Lemma~\ref{lem:charBoxSlash}, \(\perp\,=\,\boxslash\) iff for all \(a\perp b\) we have \(a\oplus b=a\vee b\). The last is characterisation of orthomodular posets among effect algebras.
Similarly, orthoalgebras are characterised by the property \(a\oplus a'=a\vee  a'=1\). 
\end{proof}
As shown in~\cite{MZ2020,My}, for each $F \in \RelFA$, the simplicial set $N(E)$ admits a $\mathbb{Z}$-action on $n$-simplices ($n \geq 0$), compatible with face and degeneracy maps, called \emph{rotation}. Orthomodular effect algebras are characterized among effect algebras via a lifting problem on \(3\)-simplices of $N(E)$. Thus we may examine its four rotations~\eqref{eq:rotationsOrthomodularLaw}, where each case, together with \(a \perp b\) and the two simplices, yields three assumptions which are written under the corresponding diagram.

\begin{equation}\label{eq:rotationsOrthomodularLaw}
\begin{array}{cccc}
        \begin{tikzcd}[row sep={5em,between origins}, column sep={5em,between origins},ampersand replacement=\&]
0\arrow[d,"a" {description}]\arrow[r]\arrow[rd,"d" {description},pos=0.75]\&2\arrow[d,"b" {description}]\\  
1\arrow[ru,dashed]\arrow[r]\&3       
        \end{tikzcd}&
          \begin{tikzcd}[row sep={5em,between origins}, column sep={5em,between origins},ampersand replacement=\&]
1\arrow[d,dashed]\arrow[r]\arrow[rd,"a'"{description},pos=0.75]\&3\arrow[d,"d'" {description}]\\  
2\arrow[ru,"b"{description},pos=0.75]\arrow[r]\&4         
        \end{tikzcd}
    &
          \begin{tikzcd}[row sep={5em,between origins}, column sep={5em,between origins},ampersand replacement=\&]
2\arrow[d,"b" {description}]\arrow[r]\arrow[rd,dashed]\&4\arrow[d,"a" {description}]\\  
3\arrow[ru,"d'"{description},pos=0.75]\arrow[r]\&5         
        \end{tikzcd}&
          \begin{tikzcd}[row sep={5em,between origins}, column sep={5em,between origins},ampersand replacement=\&]
3\arrow[d,"d'" {description}]\arrow[r]\arrow[rd,"b'"{description,pos=0.75}]\&5\arrow[d,dashed]\\  
4\arrow[ru,"a"{description},pos=0.75]\arrow[r]\&6         
        \end{tikzcd}
    \\
    a\perp b,  & b\perp d',& a\perp b,b\perp d',& d'\perp a, \\
a,b\leq d & b,d'\leq a'&d'\perp a &a,d'\leq b'
\end{array}
    \end{equation}
See that the first two and the last one lead to the same property, only the roles of \(a\), \(b\), and \(d\) (respective corresponding orthosupplements) are rotated.
While the third one leads to familiar \emph{coherence law} of orthomodular posets:
\[a\perp b,b\perp c,c\perp a\implies a\oplus b\oplus c \text{\ exists.}\]


\section{Homology group}

In the theory of effect algebras, there is an important adjunction between $\EA$ and the category of partially ordered Abelian groups with order unit $\POG_u$ (see~\cite{DvPu}):
\[
\begin{tikzcd}[column sep=huge, row sep=large]
\mathbf{EA}
  \arrow[r, bend left=30, "U"]
  \arrow[r, phantom, "\dashv" description]
&
\mathbf{POG}_u
  \arrow[l, bend left=30, "\Gamma"]
\end{tikzcd}
\]
Here \(U\) associates to an effect algebra \(E\) its universal group \((\Gamma(E),1)\), while \(\Gamma\) is the Mundici functor sending \((A,u)\in\POG_u\) to the interval effect algebra \([0,u]_A\). In the construction of the universal group, one can already recognize the first homology group.

Recall the construction of \emph{the first homology group of a simplicial set} \(X\) (see, e.g., \cite[Ch.~3]{GJ99}). For each \(n\in\mathbb{N}\), set \(C_n(X)\) to be the free Abelian group generated by \(X_n\). Equivalently, \(C_n(X)\cong \mathbb{Z}^{(X_n)}\), the group of finitely supported functions from \(X_n\) to \(\mathbb{Z}\).
The \emph{boundary map} \(\partial_n\colon C_n(X)\to C_{n-1}(X)\) is defined on generators $[a]$, $a\in X_n$, by
\[
[a]\mapsto \sum_{i=0}^n (-1)^i[d^n_i(a)].
\]
These maps satisfy \(\partial_{n-1}\circ\partial_n=0\) for all \(n\). Hence, we can define the homology group \(H_n(X)\) as the quotient \(\ker(\partial_n)/\operatorname{im}(\partial_{n+1})\).

\begin{theorem}
    Let \(E\) be an effect algebra. Then the universal group \(\Gamma(E)\) is, as an Abelian group, isomorphic to the first homology group \(H_1(N(E))\).
\end{theorem}

\begin{proof}
    Unwinding the definition of \(H_1(N(E))\): since each edge \(a\in N(E)_1\) is a loop (i.e.\ \(s(a)=t(a)\)), we have
    \(\ker(\partial_1)=C_1(N(E))\cong \mathbb{Z}^{(E)}\).
    The image of \(\partial_2\) is the subgroup of \(\mathbb{Z}^{(E)}\) generated by \([a]-[c]+[b]\) for all \(a,b,c\in N(E)_1 = E\) with \(c=a\oplus b\).
    Thus the definition of \(H_1(N(E))\) coincides with that of the universal group described in~\cite{DvPu}.
\end{proof}
\section{Mapping space}\label{sec:5}
Since $\epsilon$-simplicial sets form a presheaf category, for each pair of $\epsilon$-simplicial sets $X$ and $Y$ there is an $\epsilon$-simplicial set $[X,Y]$ (also denoted $Y^X$) such that for each $Z$ we have the adjunction
\begin{equation}\label{eq:inHom}
 \begin{split}
        Z\to [X,Y]\\
        \hline
         Z\times X\to Y.
    \end{split}
\end{equation}
By Yoneda’s lemma,
\begin{align*}
[X,Y](n)&=\Hom(\Delta^n,[X,Y])=\Hom(\Delta^n\times X,Y),\\
[X,Y](\epsilon)&=\Hom(\Sim^1,[X,Y])=\Hom(\Sim^1\times X,Y).
\end{align*}
It is well known that for a pair of (quasi-)groupoids, their mapping space is again a (quasi-)groupoid. A similar result holds for effect algebras.

\begin{theorem}\label{thm:mappingSpace}
    Let $E,F$ be effect algebras. Then there exists a Frobenius algebra $G$, a coproduct (disjoint union) of effect algebras in $\RelFA$, such that
    \begin{equation}
        N(G)\cong[N(E),N(F)].
    \end{equation}
    Moreover, for the trivial one-element effect algebra $\underline{0}$, $N(\underline{0})$ is the terminal object in $\eSSet$, hence the monoidal unit.
\end{theorem}

\begin{proof}
 We will analyze the simplices of $[N(E),N(F)]$:
 \begin{enumerate}
 \item[(i)] \emph{Vertices:}
Observe that $\Delta^0\times N(E)$ is isomorphic to $N'(E)$, which arises from $N(E)$ be forgetting  all elements in $N(E)(\eps)$. 
 Hence, the vertices of $[N(E),N(F)]$ correspond to partial monoid morphisms $E\rightarrow F$ (which may not preserve top element). 

 \item[(ii)] \emph{Edges:}  
 Given $H\colon \Delta^1\times N(E)\to N(F)$, let $h_i$, $i=0,1$, be its restriction to $\{i\}\times N(E)$. By (i), we can identify $h_0,h_1$ with partial monoid morphisms $E\to F$.  
 Let $v$ be the unique vertex of $N(E)$ and set $x=H(\Delta^1\times\{v\})\in F$. For each $a\in N(E)_1=E$, the square $\Delta^1\times\{a\}$ maps to
\begin{equation}\label{diag:internalMap}
    \begin{tikzcd}
        H(\{1\}\times \{v\})\arrow[rr,"h_1(a)"]&&H(\{1\}\times \{v\})\\
        \\
        H(\{0\}\times \{v\})\arrow[rr,"h_0(a)"]\arrow[uu,"x"]\arrow[rruu,"H(\Delta^1\times \{a\})"]&&H(\{0\}\times \{v\})\arrow[uu,"x"']
    \end{tikzcd}
\end{equation}
so $h_1(a)\oplus x=x\oplus h_0(a)$. By commutativity and cancellativity in $F$, $h_0=h_1$ and $x\leq h_0(1)'$.  
In other words, all edges in $[N(E),N(F)]$ 
are loops and $[N(E),N(F)]$ decompose as the following coproduct (disjoint union) in $\eSSet$:
\begin{equation}\label{eq:coprodEdges}
    \coprod_{h\colon E\rightarrow F} M_h.
\end{equation}
Given a vertex corresponding to partial monoid morphism $h\colon E\rightarrow F$, 
the loops over $h$ correspond to elements of $[0,h(1_E)']\subseteq F$.

\item[(iii)] \emph{$\epsilon$-edges:}  Assume $H\colon \Sim^1\times N(E)\rightarrow N(F)$. The domain of $H$ has 
a unique $\epsilon$-edge $\Sim^1\times\{1_E\}$. Since $H$ preserves $\epsilon$-edges, the diagonal edge in~\eqref{diag:internalMap} necessarily equals $1_F$. 
Using the identification of edges over $h\colon E\rightarrow F$ with elements in $[0,h(1_E)']$ from the 
previous point, we see that the $\epsilon$-edge corresponds to $h(1_E)'$.
\item[(iv)] Following the same pattern, one can prove that for each partial monoid morphism $h\colon E\rightarrow F$, the corresponding $\epsilon$-simplicial set $M_h$ from \eqref{eq:coprodEdges} coincides with $N([0,h(1)'])$.
Since each $[0,h(1)']$ is an effect algebra, the desired $G$ equals the coproduct (disjoint union) of effect algebras in $\RelFA$.
\end{enumerate}
The “moreover” part is immediate.
\end{proof}

Now follows one of our main theorems.

\begin{theorem}\label{thm:enrichedEA}
    There exists a category $\EA$, enriched over the cartesian monoidal category $\RelFA$, whose objects are effect algebras and such that for each pair $E,F$ we have
    \begin{equation}
        N(\EA(E,F))\cong[N(E),N(F)].
    \end{equation}
    Moreover, the underlying category $\EA_0$ is isomorphic to the ordinary category of effect algebras.
\end{theorem}

\begin{proof}
    Effect algebras can be viewed as a category enriched over $\SSet^\epsilon$ via the embedding $N$.  
    By Proposition~\ref{prop:monoidalFunctor}, the nerve functor $N$ is a fully faithful monoidal embedding, and for each $E,F$ the $\epsilon$-simplicial set $[N(E),N(F)]$ lies (up to isomorphism) in the image of $N$.  
    Hence the $\SSet^\epsilon$-enrichment restricts to a $\RelFA$-enrichment.

    For the “moreover” part: by Theorem~\ref{thm:mappingSpace}, a morphism $\Sim^0\to[N(E),N(F)]$ corresponds to a partial monoid morphism $h\colon E\to F$ with $[0_F,h(1_E)']\cong \underline{0}$, i.e.\ $h(1_E)=1_F$. Thus $h$ is a morphism of effect algebras.
\end{proof}


\section{Some results from simplicial combinatorics}

As observed earlier, simplicial combinatorics provides a fruitful approach to effect algebras.  
In the next section, where we consider more general $\mathrm{FA}$ than effect algebras, we show that our algebraic structures of interest can be analyzed purely in the language of simplicial combinatorics.  
As preparation, we recall some notation and establish several technical lemmas.

Given morphisms $f\colon A_0\to A_1$ and $g\colon X_0\to X_1$ in $\mathcal{C}$, the morphism
\begin{equation}\label{eq:pushoutProd}
f \boxprod g \colon A_1 \times X_0 \coprod_{A_0 \times X_0} A_0 \times X_1 \longrightarrow A_1\times X_1
\end{equation}
is called the \emph{pushout product} or \emph{box product}. 

\begin{example}\label{ex:cup}
Let $n\geq 1$. Then 
\[
(\Delta^{\{0\}}\subset\Delta^{\{0,1\}})\boxprod (\partial\Delta^n\subset\Delta^n)\in 
\overline{\{\Lambda^{n+1}_i\mid 0 \leq n\}}.
\]
We will demonstrate the proof on the case $n=2$. Then $B:=\Delta^1\times \Delta^2$ is a prism with triangular base. The domain $A$ of the box product contains the lower base of $B$ and the three faces:
\begin{equation}\label{diag:prismEx}
    \begin{tikzcd}[row sep=9pt, column sep=15pt]
& 2'  & \\
  0' 
    \arrow[rr, bend right=5] 
    \arrow[ru] 
  & & 
  1' 
    \arrow[lu] \\
  &[10pt] & \\
  & 2 \arrow[uuu] & \\
  0 \arrow[uuuur]\arrow[uuurr]
  \arrow[uuu]
    \arrow[rr, bend right=5] 
    \arrow[ru] 
  & & 
  1 \arrow[uuuul]\arrow[uuu]
    \arrow[lu] 
\end{tikzcd}
\end{equation}
We can fill the prism in three steps, by filling the horns 
$\Lambda^{\{0,1,2,2'\}}_{2}$, 
$\Lambda^{\{0,1,1',2'\}}_{1}$, and 
$\Lambda^{\{0,0',1',2'\}}_{0}$.  
Hence $A\subset B$ is a composition of three pushouts of $\Lambda^3_i\subset \Delta^3$.  
\end{example}

The box product naturally arises in the following situation.  
For morphisms $f\colon A\to B$, $g\colon X\to Y$, and $h\colon M\to N$ in a cartesian closed category, the following lifting problems are equivalent via adjunction:
\begin{equation}\label{diag:3equivalentLifProb}
    \begin{tikzcd}[row sep=10pt, column sep=13pt]
        A\arrow[d,"f"]\arrow[r]&M^Y\arrow[d]&A\times Y\coprod_{A\times X} B\times X \arrow[r]\arrow[d,hook]&M\arrow[d,"h"] & X\arrow[d,"g"]\arrow[r]&{M^B}\arrow[d]\\
         B\arrow[ru,dotted]\arrow[r]&{M^X\times_{N^X}N^Y}&B\times Y \arrow[ru,dotted]\arrow[r]& N& Y\arrow[ru,dotted]\arrow[r]&{M^A\times_{N^A}N^B}
    \end{tikzcd}
\end{equation}
In each version of~\eqref{diag:3equivalentLifProb}, exactly one of $f,g,h$ stands alone.  
This is the key idea in the proof of the following lemma:

\begin{lemma}[Rezk, Prop.~19.9.]\label{lem:boxAndClosure}
Let $\mathcal{K},\mathcal{L}$ be classes of morphisms of a cartesian closed category with pushouts. Then
\[
\overline{\mathcal{K}}\boxprod\overline{\mathcal{L}}\subseteq \overline{\mathcal{K}\boxprod\mathcal{L}}.
\]
\end{lemma}

Let $0\leq i<n$. A. Joyal showed that $\Lambda^n_i\subset \Delta^n$ is a retract of $(\Delta^{\{0\}}\subset \Delta^{\{0,1\}})\boxprod(\Lambda^n_i\subset \Delta^n)$ (see~\cite{Joyal2008}, Lem.~2.25).   
Consequently,
\begin{align*}
    (\Lambda^n_i\subset \Delta^n)\boxprod \mathrm{Cell} 
    &\subseteq \overline{(\Delta^{\{0\}}\subset \Delta^{\{0,1\}})\boxprod(\Lambda^n_i\subset \Delta^n)\boxprod\mathrm{Cell}}\\
    &\subseteq \overline{(\Delta^{\{0\}}\subset \Delta^{\{0,1\}})\boxprod \mathrm{Cell}}
    \subseteq \overline{\mathrm{Horn}}.
\end{align*}
For the case $i=n$, the same argument works with $\Delta^{\{1\}}\subset\Delta^{\{0,1\}}$.
We will consider simplicial sets, that admit fillings of $\Lambda^n_i$, only for $n\geq 3$. 
For this reason, we will need more subtle modification of the argument above.
\begin{lemma}\label{lem:joyal}
Let $1\leq n$ and $0\leq k$. Then
\begin{equation}\label{eq:joyal}
    \mathrm{Horn}_{\geq n}\boxprod\mathrm{Cell}_{\geq k}\subseteq \overline{\mathrm{Horn}}_{\geq\max\{k+1,n\}}.
\end{equation}
\end{lemma}
\begin{proof}
Assume $n\leq m$ and $0\leq i\leq m$. Depending on $i$, let $\Delta^0\subset\Delta^1$ be one of the two inclusions.  
By the retract trick:
\begin{equation}\label{eq:joyalTrick}
    (\Lambda^m_i\subset \Delta^m)\boxprod \mathrm{Cell}_{\geq k} 
    \subseteq \overline{(\Delta^0\subseteq \Delta^1)\boxprod(\Lambda^m_i\subset \Delta^m)\boxprod\mathrm{Cell}_{\geq k}}.
\end{equation}  
For $l\geq k$, let $A\subset B$ denote
\[
\Delta^m\times \partial\Delta^l\coprod_{\Lambda^m_i\times\partial\Delta^l}\Lambda^m_i\times\Delta^l\subset \Delta^m\times \Delta^l.
\]
Then $A$ contains all $j$-simplices of $B$ with $j\leq l-1$ or $j\leq m-2$.  
Thus~\eqref{eq:joyalTrick} is a subclass of
\[
\overline{(\Delta^0\subset\Delta^1)\boxprod \mathrm{Cell}_{\geq\max\{l,m-1\}}}\subseteq \overline{(\Delta^0\subset\Delta^1)\boxprod \mathrm{Cell}_{\geq\max\{k,n-1\}}}.
\]
Using Example~\ref{ex:cup}, the claim follows.
\end{proof}

\begin{lemma}\label{lem:essentialHelp}
Let $n\geq 3$, $I\subseteq \{0,\ldots,n\}$, and $\mathcal{K}$ the weakly saturated class generated by
\begin{equation}\label{eq:K}
\{\Lambda^k_i\subset \Delta^k\mid 3\leq k,\, 0\leq i\leq k\}\cup \{\partial_0\Delta^3\cup\partial_2\Delta^3\subset\Delta^3,\, \partial_1\Delta^3\cup\partial_3\Delta^3\subset\Delta^3\}.
\end{equation}
Then $\cup_{i\in I}\partial_i\Delta^n\subseteq \Delta^n$ belongs to $\mathcal{K}$ whenever
\begin{enumerate}
    \item[(i)] $3\leq |I|\leq n$, or
    \item[(ii)] $|I|=2$ and the two indices of $I$ are not consecutive modulo $n+1$.
\end{enumerate}
\end{lemma}
\begin{proof}
(i)  By induction on $n$. The base case $n=3$ is trivial. Let $n>3$ and $j\in \{0,\ldots,n\}\setminus I$. We can apply the induction hypothesis to $\partial_j\Delta^n$ and obtain that $$\bigcup_{i\in I}\partial_i\Delta^n\subset \bigcup_{i\in I\cup\{j\}}\partial_i\Delta^n$$
belongs to $\mathcal{K}$. Repeating this process, we reach $|I|=n$, in which case we obtain a horn of the form $\Lambda^n_j\subset\Delta^n$, $j\leq n$, which belongs to $\mathcal{K}$.

(ii) Again use induction on $n$. The case $n=3$ is trivial. Let $n>3$, and let $i<j$ be the indices of $I$.  
If $j-i\geq 2$, apply induction to $\partial_{i+1}\Delta^n$, reducing to case (i).  
If $j=i+2$, choose some $k$ not between $i$ and $j$, and apply induction to $\partial_k\Delta^n$. The rest is similar.
\end{proof}

\begin{lemma}\label{lem:genHorn}
Let $n,k\in\mathbb{N}$ and $I\subset\{0,\ldots,n\}$ with $k=|I|\leq n$. Then 
\[
\bigcup_{i\in I}\partial_i\Delta^n \subset \Delta^n \in \overline{\mathrm{Horn}}_{\geq k}.
\]
\end{lemma}
\begin{proof}
This is a straightforward generalization of the proof of Lemma~\ref{lem:essentialHelp}(i). 
\end{proof}
\begin{lemma}\label{lem:lotsOfWork}
Let $\mathcal{K}$ be as in Lemma~\ref{lem:essentialHelp} and $i,j\in\{0,1,2\}$.  
Then, except in the case $i=j=1$, the inclusion
\begin{equation}
    \Lambda^2_i\times\Delta^2\coprod_{\Lambda^2_i\times \Lambda^2_j}\Delta^2\times \Lambda^2_j\subseteq \Delta^2\times\Delta^2
\end{equation}
belongs to $\mathcal{K}$.
\end{lemma}

\begin{proof}
Since the generating set of $\mathcal{K}$ is symmetric under arrow reversal, and the pushout product has an obvious symmetry, it suffices to treat the three cases $i=0,1,2$ with $j=2$.
\smallskip
\emph{Case $i=1$.}  
Let $m\colon A\hookrightarrow B$ denote the inclusion in question.  
Here $B=\Delta^2\times\Delta^2$ is the union of six $4$-simplices.  
We label its vertices as in diagram~\eqref{diag:prismRelabel} (only some edges are shown):
\begin{equation}\label{diag:prismRelabel}
    \begin{tikzcd}[row sep=12pt, column sep=15pt]
2\arrow[rrr]&&& 2'\arrow[rrr]&&& 2''&\\
& 1\arrow[lu]\arrow[rrr]&&& 1'\arrow[lu]\arrow[rrr]&&& 1''\arrow[lu]\\
0\arrow[uu,bend left=5]\arrow[ru]\arrow[rrr]&&& 0'\arrow[uu,bend left=5]\arrow[ru]\arrow[rrr]&&& 0''\arrow[ru]\arrow[uu,bend left=5]&\\
\end{tikzcd}
\end{equation}
Here $\Delta^{\{0,1\}}\times\Delta^2\subset A$ and $\Delta^{\{1,2\}}\times\Delta^2\subset A$ are the left and right prism, respectively.  

We now form the adjacency graph~\eqref{diag:adjanGr}, whose vertices are the six non-degenerate $4$-simplices of $B$, with edges indicating a shared $3$-face.  
A vertex is marked whenever the opposite $3$-face lies in $A$: 
\begin{equation}\label{diag:adjanGr}
    \begin{tikzcd}[row sep=15pt, column sep=7pt]
\Delta^{\{\markindex{0},\markindex{1},2,2',\markindex{2''}\}}\arrow[r,no head]&\Delta^{\{\markindex{0},1,1',2',\markindex{2''}\}}\arrow[d,no head]\arrow[r,no head]&\Delta^{\{\markindex{0},1,1',1'',2''\}}\arrow[d,no head]&\\
& \Delta^{\{\markindex{0},0',\markindex{1'},2',\markindex{2''}\}}\arrow[r,no head]&\Delta^{\{\markindex{0},0',1',1'',2''\}}\arrow[r,no head]&\Delta^{\{\markindex{0},0',0'',\markindex{1''},2''\}}
\end{tikzcd}
\end{equation}

Label these simplices $D_1,\dots,D_6$ in the order indicated by~\eqref{diag:adjanGr}. To improve orientation in the diagram, let us describe in detail which faces of a given $D_l$ belong to $A$. Some $\Delta^I\subset D_l$ belongs to $A$ if:
\begin{itemize}
    \item $I$ omits $0,0',0''$, or
    \item $I$ omits $1,1',1''$, or
    \item $I$ omits $0,1,2$, or
    \item $I$ omits $0'',1'',2''$.
\end{itemize}
Note also that for each of the four conditions there is a unique greatest face $\Delta^I\subset D_l$, satisfying the condition.

The domain $A$ is the union
\[
A=\Delta^2\times \Delta^{\{0,2\}} \;\cup\; \Delta^2\times \Delta^{\{1,2\}} \;\cup\; \Delta^{\{0,1\}}\times\Delta^2 \;\cup\; \Delta^{\{1,2\}}\times \Delta^2,
\]
consisting of $12$ non-degenerate $3$-simplices, exactly matching the $12$ marked indices.  
Thus the graph~\eqref{diag:adjanGr} encodes $A$ completely.

We now fill the $D_l$'s sequentially using Lemma~\ref{lem:essentialHelp}.  
Simplex $D_1$ intersects $A$ in three faces, hence we can fill it. 
The adjanced simplex $D_2$ then has three filled faces, so we can fill it as well. This provides the face of $D_3$ opposite the vertex $1''$, 
which allows us to fill $D_3$. The simplex $D_4$ then has four filled faces (one arising from $D_2$). 
After filling $D_4$, the simplex $D_5$ admits three filled faces, and can thus be filled. 
Finally, the last simplex has three filled faces at this stage, so we conclude with the final application of Lemma~\ref{lem:essentialHelp}.

Note that at each step, the intersection of $D_l$ with $A$ lies inside the union of its already filled $3$-faces, so the use of Lemma~\ref{lem:essentialHelp} is valid.

\smallskip
\emph{Cases $i=0$ and $i=2$.}  
The argument is entirely analogous (in
fact, there are many possible ways to proceed, so we feel confident leaving the
details to the reader).  
The marked vertices in the adjacency graph are modified as in diagrams~\eqref{diag:prismCase4Graph0} and~\eqref{diag:prismCase4Graph2}, respectively:

\begin{equation}\label{diag:prismCase4Graph0}
    \begin{tikzcd}[row sep=15pt, column sep=7pt]
\Delta^{\{\markindex{0},\markindex{1},2,\markindex{2'},\markindex{2''}\}}\arrow[r,no head]&\Delta^{\{\markindex{0},1,1',2',\markindex{2''}\}}\arrow[d,no head]\arrow[r,no head]&\Delta^{\{\markindex{0},1,\markindex{1'},1'',2''\}}\arrow[d,no head]&\\
& \Delta^{\{0,0',\markindex{1'},2',\markindex{2''}\}}\arrow[r,no head]&\Delta^{\{0,0',1',1'',2''\}}\arrow[r,no head]&\Delta^{\{0,\markindex{0'},0'',\markindex{1''},2''\}}
\end{tikzcd}
\end{equation}

\begin{equation}\label{diag:prismCase4Graph2}
    \begin{tikzcd}[row sep=15pt, column sep=7pt]
\Delta^{\{\markindex{0},\markindex{1},2,\markindex{2'},2''\}}\arrow[r,no head]&\Delta^{\{\markindex{0},1,1',2',2''\}}\arrow[d,no head]\arrow[r,no head]&\Delta^{\{\markindex{0},1,\markindex{1'},1'',2''\}}\arrow[d,no head]&\\
& \Delta^{\{\markindex{0},0',\markindex{1'},2',2''\}}\arrow[r,no head]&\Delta^{\{\markindex{0},0',1',1'',2''\}}\arrow[r,no head]&\Delta^{\{\markindex{0},\markindex{0'},0'',\markindex{1''},2''\}}
\end{tikzcd}
\end{equation}

This completes the proof.
\end{proof}


We remark that in the case $i=j=1$, each simplex in the corresponding adjancency graph has both its first and last vertex marked, so Lemma~\ref{lem:essentialHelp} does not apply.

For later use, we require a variant of the previous lemma in which $\Lambda^2_j$ is replaced by $\Delta^{\{0,2\}}\cup\Delta^{\{1\}}$.

\begin{lemma}\label{lem:lotsOfWork2}
Let $\mathcal{K}$ be as in Lemma~\ref{lem:essentialHelp}.  
For $i=0,1,2$, the morphism
\begin{equation}
\Lambda^2_i\times\Delta^2\;\coprod_{\Lambda^2_i\times (\Delta^{\{0,2\}}\cup\Delta^{\{1\}})}\;\Delta^2\times (\Delta^{\{0,2\}}\cup\Delta^{\{1\}})
\;\subseteq\; \Delta^2\times\Delta^2
\end{equation}
belongs to $\mathcal{K}$.
\end{lemma}

\begin{proof}
We use the notation from Lemma~\ref{lem:lotsOfWork}.  
By symmetry, it suffices to treat $i=1$ and $i=0$.

\smallskip
\emph{Case $i=1$.}  
Here $A$ is the union of nine $3$-simplices and one $2$-simplex $\Delta^2\times\Delta^{\{1\}}$, corresponding in diagram~\eqref{diag:prismRelabel} to $\Delta^{\{1,1',1''\}}$.  
The adjacency graph is given in~\eqref{diag:adjGr4}.  
A simplex $\Delta^I\subset D_l$ belongs to $A$ iff:  
(1) $I$ omits $1,1',1''$; (2) $I$ omits $0,1,2$; or (3) $I$ omits $0'',1'',2''$.
    \begin{equation}\label{diag:adjGr4}
    \begin{tikzcd}[row sep=15pt, column sep=7pt]
\Delta^{\{0,\markindex{1},2,2',\markindex{2''}\}}\arrow[r,no head]&\Delta^{\{0,1,1',2',\markindex{2''}\}}\arrow[d,no head]\arrow[r,no head]&\Delta^{\{0,1,1',1'',2''\}}\arrow[d,no head]&\\
& \Delta^{\{\markindex{0},0',\markindex{1'},2',\markindex{2''}\}}\arrow[r,no head]&\Delta^{\{\markindex{0},0',1',1'',2''\}}\arrow[r,no head]&\Delta^{\{\markindex{0},0',0'',\markindex{1''},2''\}}
\end{tikzcd}
\end{equation}

We fill $D_1, D_4$, and $D_6$ first, since each intersects $A$ in a union of $3$-faces, Lemma~\ref{lem:essentialHelp} applies.  
Next, $D_2$ has filled faces opposite $1,1',2''$, and $D_5$ has filled faces opposite $0,1',1''$, so both can be filled.  

It remains to handle $D_3$.  
The adjacent simplices provide the faces opposite $1$ and $1''$, while $\Delta^{\{1,1',1''\}}\subset A\cap D_3$.  
Restricting to $\partial_{2''}D_3\cong\Delta^3$, these three faces form a horn, which can be filled.  
This gives the face $\partial_{2''} D_3$, so $D_3$ can then be filled by Lemma~\ref{lem:essentialHelp}.

\smallskip
\emph{Case $i=0$.}  
Here $\Delta^I\subseteq D_l$ lies in $A$ iff:  
(1) $I$ omits $1,1',1''$; (2) $I$ omits $0',1',2'$; or (3) $I$ omits $0'',1'',2''$.  
The modified adjacency graph is shown in~\eqref{diag:adjGr5}.
\begin{equation}\label{diag:adjGr5}
    \begin{tikzcd}[row sep=15pt, column sep=7pt]
\Delta^{\{0,\markindex{1},2,\markindex{2'},\markindex{2''}\}}\arrow[r,no head]&\Delta^{\{0,1,1',2',\markindex{2''}\}}\arrow[d,no head]\arrow[r,no head]&\Delta^{\{0,1,\markindex{1'},1'',2''\}}\arrow[d,no head]&\\
& \Delta^{\{0,0',\markindex{1'},2',\markindex{2''}\}}\arrow[r,no head]&\Delta^{\{0,0',1',1'',2''\}}\arrow[r,no head]&\Delta^{\{0,\markindex{0'},0'',\markindex{1''},2''\}}
\end{tikzcd}
\end{equation}  

We first fill $D_1$, $D_4$, and $D_6$.  
Then we can fill $D_2$. Consider $D_3$, it has already filled faces $\partial_{1'}D_3$, $\partial_{1''}D_3$, and $\Delta^{\{1,1'',1'''\}}$. Restricting to $\partial_{2''}D_3$, we obtain a $3$-horn, which we can fill. So we obtain the face $\partial_{2''}D_3$ and those we can apply Lemma~\ref{lem:essentialHelp}.
Finally, after obtaining three filled faces of $D_5$, we fill it by Lemma~\ref{lem:essentialHelp}.  
\end{proof}

We will only need the next lemma for $m,n\leq 2$, which can be proved directly in the style of Lemma~\ref{lem:lotsOfWork}.  
Nevertheless, we state it in general to highlight a useful correspondence between top-dimensional simplices of $\Delta^m\times\Delta^n$ and lattice paths in the $m\times n$ grid.  
We denote the domain of $(\partial\Delta^n\subseteq\Delta^n)\Box(\partial\Delta^m\subset\Delta^m)$ by $\partial(\Delta^m\times\Delta^n)$.

\begin{lemma}\label{lem:hope}
    Let $n,m\in\mathbb N$. Then $(\partial\Delta^n\subseteq\Delta^n)\Box(\partial\Delta^m\subset\Delta^m)$ belongs to the weakly saturated class generated by
    horns $\Lambda^k_i$, with $k\geq \max\{m,n\}+1$ and the cell $\partial\Delta^l\subset\Delta^{m+n}$.

Consequently, $$\overline{\Cell}_{\geq m}\boxprod \overline{\Cell}_{\geq n}\subseteq \overline{\Horn_{\geq \max\{m,n\}+1}\cup\Cell_{\geq m+n}}.$$
\end{lemma}
\begin{proof}
    Without loss of generality assume $n\leq m$. Let us organize the vertices of $\Delta^m\times\Delta^n$ into an $m\times n$ grid, 
    labeling the vertices with elements of $\{0,\ldots,n\}\times\{0,\ldots,m\}$.
    There is a one-to-one correspondence between top-dimensional simplices of $\Delta^m\times\Delta^n$ and the shortest lattice paths in the grid from $(0,0)$ to $(m,n)$. 
    We introduce some notation. For a top-dimension simplex $\sigma$ of $\Delta^m\times\Delta^n$, denote by $u(\sigma)$ the number of squares in the grid located above the lattice path corresponding to $\sigma$.
    Moreover, we decompose the set of $m+n+1$ vertices of $\sigma$ into three disjoint sets $I(\sigma)$, $U(\sigma)$, and $L(\sigma)$ in the following manner:
    \begin{itemize}
     \item The vertex $(i,j)$ of the path belongs to $I(\sigma)$ if and only if it is the first, the last, or a straight vertex of the path; that is, there is a vertical or horizontal line of the grid that meets the path exactly at the vertex $(i,j)$.
    \item The vertex $(i,j)$ of the path belongs to $U(\sigma)$ if and only if there is a north–east turn at $(i,j)$.
    \item The vertex $(i,j)$ of the path belongs to $L(\sigma)$ if and only if there is an east–north turn at $(i,j)$.
   \end{itemize}
Consider the example where $m=4$, $n=2$, and $\sigma$ is defined by the thick path in diagram~\eqref{diag:grid}. Here, $u(\sigma)=5$, and the sets $I(\sigma)$, $U(\sigma)$, and $L(\sigma)$ are indicated in the diagram.
\begin{equation}\label{diag:grid}
    \begin{tikzcd}[row sep=10pt, column sep=20pt]
        (0,2) \arrow[r] & (1,2) \arrow[r]& (2,2)\arrow[r]& U\arrow[r,Tarrow]&I\\
        (0,1) \arrow[r]\arrow[u] & (1,1) \arrow[r]\arrow[u]&U\arrow[r,Tarrow]\arrow[u]& L\arrow[r]\arrow[u,Tarrow]&(4,1)\arrow[u]\\
        I \arrow[r,Tarrow]\arrow[u] &I\arrow[u]\arrow[r,Tarrow]& L\arrow[r]\arrow[u,Tarrow]&(0,3)\arrow[u]\arrow[r]\arrow[u]&(0,4)\arrow[u]
    \end{tikzcd}
\end{equation}

   Let us order all the top-dimensional simplices of $\Delta^m \times \Delta^n$ as 
$\sigma_1, \sigma_2, \ldots, \sigma_N$ in such a way that for every pair 
$\sigma_i, \sigma_j$ with $i \leq j$, we have $u(\sigma_i) \leq u(\sigma_j)$. 
We claim that for each $\sigma_p$, the following holds:
    \begin{equation}\label{eq:intersectionxxx}
\sigma_p\cap\bigg(\partial(\Delta^m\times\Delta^n) \cup \sigma_1\cup \ldots \cup \sigma_{p-1}\bigg)=\bigcup_{x\in I(\sigma)\cup L(\sigma)}\partial_x\sigma_p.
\end{equation} 
For each $i \leq m$, the intersection of $\partial_i \Delta^m \times \Delta^n$ with $\sigma_p$ is the face of $\sigma_p$ obtained by omitting all the vertices lying on the line $\ell = (i,-)$. 
There is at least one vertex $x$ in $I(\sigma_p)$ or $L(\sigma_p)$ on this line. Hence, $\sigma_p \cap (\partial_i \Delta^m \times \Delta^n)$ is contained in the right-hand side of~\eqref{eq:intersectionxxx}. An analogous argument applies to $\Delta^m \times \partial_i \Delta^n$, so $\sigma_p\cap\partial(\Delta^m \times \Delta^n)$ is contained in the right-hand side.

By the way we have ordered the simplices $\sigma_q$, for each $\sigma_q$ with $q < p$, there is at least one vertex $x \in L(\sigma_p)$ that is not a vertex of $\sigma_q$. 
Hence, $\sigma_q \cap \sigma_p \subset \partial_x \sigma_p$, which establishes the left-hand inclusion of~\eqref{eq:intersectionxxx}. 
On the other hand, for each vertex $x \in I(\sigma_p)$, the face $\partial_x \sigma_p$ is clearly contained in $\partial(\Delta^m \times \Delta^n)$. 
For each $x \in L(\sigma_p)$, there is a unique $\sigma_q$ such that $u(\sigma_q) = u(\sigma_p) - 1$ and $\partial_x \sigma_p = \sigma_p \cap \sigma_q$.
Thus, we have proved~\eqref{eq:intersectionxxx}.

Note that for each $\sigma_p$, the number of vertices in $U(\sigma_p)$ is at most $n$, so the intersection~\eqref{eq:intersectionxxx} contains at least $(m+n+1) - n = m+1$ codimension-$1$ faces of $\sigma_p$. 
It follows that we can fill the simplices $\sigma_1, \ldots, \sigma_{N-1}$ one by one aplying fillings of $\mathrm{Horn}_{\geq m+1}$. 
Finally, we fill $\sigma_N$ using the filling of $\partial \Delta^{m+n} \subset \Delta^{m+n}$ (note that $U(\sigma_N) = \emptyset$).
    \end{proof}

\section{Non-commutative effect algebras}
Given two effect algebras $E$ and $F$, the mapping space $[N(E),N(F)]$ described in (the proof of) Theorem~\ref{thm:mappingSpace} exhibits no interaction between distinct vertices.  This is a consequence of the commutativity. In contrast, the non-commutative case is more subtle: 
interactions between distinct vertices arise, and the 
structure of the mapping space becomes significantly more involved.

The following is a non-commutative analogue of an effect algebra.
\begin{definition}[\cite{DV2001}]\label{def:PEA}
    A pseudo effect algebra is a partial algebra $(E,\oplus,0,1)$, where
    \begin{enumerate}[label=(\arabic*)]
        \item $(E,\oplus,0)$ is a partial monoid;
        \item for all $a,b\in E$, if $a\perp b$, 
        then there exist $a_1,b_1\in E$ such that $b_1\oplus a = a\oplus b = b\oplus a_1$;\label{it:braiding}
        \item for each $a\in E$, there are unique elements $a^-,a^\sim$ with $a^-\oplus a=1$ and $a\oplus a^\sim=1$;\label{it:leftRingtOrthosup}
        \item for each $a\in E$, if $a\oplus 1$ or $1\oplus a$ is defined, then $a=0$. 
    \end{enumerate}
\end{definition}

Pseudo effect algebras admit a partial order and satisfy cancellation property (see~\cite{DV2001} for details). Moreover, for each $a\in E$,
\begin{equation}\label{eq:nonComOrthosupplement}
    \{b\in E\mid b\perp a\}=[0,a^-].
\end{equation}

We call property~\ref{it:braiding} from Definition~\ref{def:PEA} \emph{braiding}. It is equivalent to the statement that $N(E)$ has the right lifting property with respect to the following two subsets of the middle cylinder:
\begin{equation}\label{diag:braiding}
    \begin{tikzcd}[row sep=10pt, column sep=20pt]
        1&1&                &&  1\arrow[r,"a_2"{description}]&1 &&& 1\arrow[r,"a_2"{description}]&1\\
                  & &{}\arrow[r,hook]&{}&            &  &{}&{}\arrow[l,hook']& &\\
         0\arrow[uu,"b"{description}]\arrow[r,"a_1"{description}]\arrow[uur]&0\arrow[uu,"b"{description}] &&& 0\arrow[uu,"b"{description}]\arrow[r,"a_1"{description}]\arrow[uur]&0\arrow[uu,"b"{description}] &&& 0\arrow[uu,"b"{description}]\arrow[uur]&0\arrow[uu,"b"{description}]
    \end{tikzcd}
\end{equation}
For $F\in\RelFA$, if $N(F)$ satisfies the RLP with respect to~\eqref{diag:braiding}, we say that $F$ is \emph{braided}, or equivalently, that \emph{$F$ admits a braiding}.

\medskip

We now show that braiding, together with cancellation, yields a well-behaved notion of conjugation. 
Let $E,F\in \PEA$ and let $f,g\colon E\to F$ be morphisms of partial monoids. 
We say that $f$ and $g$ are \emph{conjugate} if there exists $b\in F$ such that 
\[
  b\oplus f(a) \;=\; g(a)\oplus b \qquad \text{for all } a\in E.
\]

\begin{lemma}\label{lem:conjugate}
Let $E$ and $F$ be pseudo effect algebras, and let $f\colon E\to F$ be a morphism of partial monoids. 
Then for each $b\in [0,f(1)^{-}]$ there exists a unique morphism 
$g\colon E\to F$
such that 
$b\oplus f(a)=g(a)\oplus b$ for all $a\in E$.
\end{lemma}

\begin{proof}
    Let $b\in[0,f(1)^{-}]$. For each $a\in E$, there is a unique $g(a)$ such that $b\oplus f(a)=g(a)\oplus b$. 
    Clearly, $g(0)=0$. Next, if $a_1,a_2\in E$ with $a_1\perp a_2$, then
    \begin{equation}\label{eq:conjugation1}
        b\oplus f(a_1\oplus a_2) = b\oplus f(a_1)\oplus f(a_2) 
        = g(a_1)\oplus b\oplus f(a_2) 
        = g(a_1)\oplus g(a_2)\oplus b,
    \end{equation}
    while
    \begin{equation}\label{eq:conjugation2}
        b\oplus f(a_1\oplus a_2) = g(a_1\oplus a_2)\oplus b.
    \end{equation}
    By cancellation, $g(a_1\oplus a_2)=g(a_1)\oplus g(a_2)$. Thus $g$ is a morphism of partial monoids.
\end{proof}

The proof that the pointwise-defined conjugate \(g\) is a morphism has a natural presentation in simplicial combinatorics. 
Consider the prism in~\eqref{diag:prism}, where the base encodes 
\(f(a) = f(a_1) \oplus f(a_2)\), and the three faces encode \(g(a_1)\), \(g(a_2)\), and \(g(a)\), are the conjugates of \(f(a_1)\), \(f(a_2)\), and \(f(a)\), respectively.
\begin{equation}\label{diag:prism}
    \vcenter{\hbox{%
    \begin{tikzcd}[row sep=12pt, column sep=30pt]
& 2'  & \\
  0' 
    \arrow[rr, "g(a_2)"{description}, bend right=5,near start=0.8] 
    \arrow[ru, "g(a)"{description}] 
  & & 
  1' 
    \arrow[lu,"g(a_1)"'{description}] \\
  &[20pt] & \\
  & 2 \arrow[uuu, "b"{description},near start=0.2] & \\
  0 \arrow[uuuur]\arrow[uuurr]
  \arrow[uuu, "b"{description}]
    \arrow[rr, "f(a_2)"{description}, bend right=5] 
    \arrow[ru, "f(a)"{description}] 
  & & 
  1 \arrow[uuuul]\arrow[uuu, "b"{description}]
    \arrow[lu, "f(a_1)"{description}] 
\end{tikzcd}
}}
\end{equation}

Since \(F\) is cancellative, it admits fillings of all \(3\)-horns. Hence we can fill the cup-shaped sub-simplicial set in three steps (see Example~\ref{ex:cup}) and obtain the top base, which gives the desired equation 
$g(a)=g(a_1)\oplus g(a_2)$.

Recalling Diagram~\eqref{diag:internalMap} from the proof of Theorem~\ref{thm:mappingSpace}, we deduce that conjugations between monoid morphisms 
\(
E \rightarrow F
\)
correspond exactly to edges in the mapping space
\(
[N(E),N(F)].
\) 
Given a morphism \(f\colon E \rightarrow F\) between pseudo effect algebras, Lemma~\ref{lem:conjugate} shows that all edges from the vertex corresponding to \(f\) are 
determined by the element \(f(1)\). This follows directly from Corollary~\ref{cor:KanFib}, which states that the evaluation morphism
\[
p\colon [E,F] \rightarrow [\underline{1},F],
\] 
induced by the unique morphism \(\underline{1} \rightarrow E\) in \(\PEA\) (where \(\underline{1}\) is initial), is a Kan fibration. 
The following proposition provides the key step for the proof.


\begin{proposition}\label{prop:bigFun}
    Let $\mathcal{L}$ and $\mathcal{M}$ be weakly saturated classes where
    \begin{itemize}
        \item $\mathcal{L}$ is generated by all horns $\Lambda^i_n$, $0\leq i\leq n$, and $\partial\Delta^k\subset\Delta^k$, $k\geq 2$;
        \item $\mathcal{M}$ is generated by~\eqref{eq:K} and braiding~\eqref{diag:braiding}.
    \end{itemize}
 Moreover, let $X\subseteq Y\in\SSet$ be a pair of simplicial sets satisfying:
\begin{enumerate}
    \item[(1)] $Y$ has a unique vertex (so that all edges are loops);
    \item[(2)] the edges in $X_1$ form a dominating set in $Y_1$ with respect to the preorder $\leq$ on $Y_1$, defined as the transitive closure of the relation
    \begin{eqnarray}\label{eq:defOfPreorder}
        a\leq b\quad\text{iff}\quad \exists c\in E, \mu\colon(a,c)\dashmapsto b\quad \text{or}\quad \mu\colon(c,a)\dashmapsto b.
    \end{eqnarray}
\end{enumerate}
Then
    \begin{equation}
        \mathcal{L}\boxprod (X\hookrightarrow Y)\subseteq \mathcal{M}. 
    \end{equation}
\end{proposition}
Before proving the proposition, we state the corollary promised above.
\begin{corollary}\label{cor:KanFib}
    Let $E$ and $F$ be a pair of pseudo effect algebras, and let $\underline{1}\rightarrow E$ be the unique morphism in $\PEA$. 
    Then 
    $$
        p\colon [E,F] \longrightarrow [\underline{1},F]
    $$ 
    is a minimal Kan fibration. Moreover, $p$ admits a unique RLP with respect to:
    \begin{enumerate}
        \item[(1)] $\partial\Delta^n \subset \Delta^n$, $n \geq 2$;
        \item[(2)] $\Delta^1 \subset \Sim^1$.
    \end{enumerate}
\end{corollary}
\begin{proof}
Let $\mathcal{L}$ and $\mathcal{M}$ be as in Proposition~\ref{prop:bigFun}. 
Consider the inclusion $i\colon N(\underline{1})\hookrightarrow N(E)$, and note that $X=N(\underline{1})$ and $Y=N(E)$ satisfy conditions (1) and (2) of Proposition~\ref{prop:bigFun}. 
We aim to show that the morphism 
\[
p\colon [N(E),N(F)] \longrightarrow [N(\underline{1}),N(F)]
\] 
has the unique right lifting property with respect to every map $f\in \mathcal{L}$.

For each $f\in\mathcal{L}$, the corresponding lifting problems can be expressed equivalently as:
\begin{equation}
\begin{tikzcd}
    {\cdot} \arrow[r]\arrow[d,"f"'] & {[N(E),N(F)]\arrow[d,"p"]} &&& {\cdot}\arrow[d,"f\boxprod i"']\arrow[r] & {N(f)}\\
    {\cdot}\arrow[ru,dotted] \arrow[r] & {[\underline{1},N(F)]} &&& {\cdot}\arrow[ur,dotted] &
\end{tikzcd}
\end{equation}
By Proposition~\ref{prop:bigFun}, we have $f\boxprod i \in \mathcal{M}$, and by assumption, $N(F)$ admits a unique lifting property with respect to $\mathcal{M}$. 

It remains to verify property (2). Expanding the definitions, we observe that
\begin{equation}
    (\Delta^1 \subset \Sim^1)\boxprod (\underline{1}\subset N(E)) = \Delta^1 \times N(E) \coprod_{\Delta^1 \times \underline{1}} \Sim^1 \times \underline{1} \subseteq \Sim^1 \times N(E),
\end{equation}
which is in fact the identity, since $\Sim^1 \times \underline{1}$ contains all the relevant $\epsilon$-edges.
\end{proof}
\begin{proof}[Proof of Proposition~\ref{prop:bigFun}]
Consider the following inclusions:
\begin{enumerate}
    \item[(a)] $g_i\colon (\Lambda^2_i)_{\mathrm{red}} \hookrightarrow \Delta^2_{\mathrm{red}}$, for $i=0,2$;
    \item[(b)] $g_1\colon \Delta^{\{0,2\}}_{\mathrm{red}} \hookrightarrow \Delta^3_{\mathrm{red}}$;
    \item[(c)] $b_k\colon \partial\Delta^k \hookrightarrow \Delta^k$, for $k\geq 2$. 
\end{enumerate}

We claim that the inclusion $X\subseteq Y$ arises as a transfinite composition of pushouts of the morphisms $g_0,g_1,g_2$ and the $b_k$'s. By the assumption that edges in $X$ form a dominating subset of $Y_1$, there exists a simplicial set $Y'$ with $X \subseteq Y' \subseteq Y$ such that $Y'$ contains all edges of $Y$, and $X\subset Y'$ belongs to the weakly saturated class generated by $g_0,g_1,g_2$. 
By Lemma~\ref{lem:fillingCells}, the inclusion $Y'\subseteq Y$ belongs to the weakly saturated class generated by $g_k$'s, $k\geq 2$.

Our goal is to show that for each $g$ from (a--c), we have $\mathcal{L}\boxprod g \subseteq \mathcal{M}$. We discuss the cases (a--c) separately.

\smallskip
\emph{Case $g_i$, $i=0,2$:} Both cases are analogous; we focus on $g_2$. Note that $g_2$ is a pushout of $\tilde{g}_2\colon \Lambda^2_2 \subset \Delta^2$ due to Lemma~\ref{lem:reduction}. 

\begin{itemize}
    \item $g_2\boxprod (\Lambda^1_j \subset \Delta^1)$, $j=1$: the domain of the morphism in question has the following shape:
    \begin{equation}\label{diag:prismCase2}
    \begin{tikzcd}[row sep=10pt, column sep=20pt]
    & 0'  & \\
      0' \arrow[ru,"c'"{description}] & & 0' \arrow[lu,"b'"{description}] \\
      & & \\
      & 0 \arrow[uuu, "d"{description}] & \\
      0 \arrow[uuuur] \arrow[uuu, "d"{description}] \arrow[rr, "a"{description}, bend right=5] 
      \arrow[ru, "c"{description}] & & 0 \arrow[uuu, "d"{description}] \arrow[lu, "b"{description}] \arrow[luuuu]
    \end{tikzcd}
    \end{equation}
    It is easy to see that this shape can be completed into the whole prism using $3$-horn fillings, associativity, and braiding. The case $j=0$ is analogous.

    \item $g_2\boxprod (\Lambda^2_i \subset \Delta^2)$, $i=0,1,2$: Since $g_2$ is a pushout of $\tilde{g}_2$, it suffices to show $\tilde{g}_2\boxprod (\Lambda^2_i \subset \Delta^2) \in \mathcal{M}$, which follows from Lemma~\ref{lem:lotsOfWork}.

    \item $g_2\boxprod (\Lambda^n_i \subset \Delta^n)$, $n\geq 3$: Since $\Lambda^n_i\subset\Delta^n$ belongs to $\overline{\mathrm{Cell}}_{\geq 2}$, Lemma~\ref{lem:joyal} yields $\tilde{g}_2\boxprod (\Lambda^n_i \subset \Delta^n) \in \mathcal{M}$.

    \item $g_2\boxprod (\partial\Delta^n \subset \Delta^n)$, $n\geq 2$: This follows from Lemma~\ref{lem:joyal} applied to $\tilde{g}_2$, since $n+1\geq 3$.
\end{itemize}
\smallskip
\emph{Case $g_1$:} Note that $g_1$ is a pushout of $\tilde{g}_1\colon \Delta^{\{0,1\}}\cup \Delta^{\{1\}} \subset \Delta^2$ by Lemma~\ref{lem:reduction}. 

\begin{itemize}
    \item $g_1\boxprod (\Lambda^1_j)$: For $j=1$, the domain of the monomorphism has the following shape:
    \begin{equation}\label{diag:prismCase3}
    \begin{tikzcd}[row sep=10pt, column sep=20pt]
    & 0'  & \\
      0' \arrow[ru,"c'"{description}] & & 0' \\
      & & \\
      & 0 \arrow[uuu, "d"{description}] & \\
      0 \arrow[uuuur] \arrow[uuu, "d"{description}] \arrow[rr, "a"{description}, bend right=5] 
      \arrow[ru, "c"{description}] & & 0 \arrow[uuu, "d"{description}] \arrow[lu, "b"{description}] 
    \end{tikzcd}
    \end{equation}
    Again, using associativity (twice), braiding (twice), and $3$-horn fillings, we can complete the prism.

    \item $g_1\boxprod (\Lambda^2_i \subset \Delta^2)$, $i=0,1,2$: Since $g_1$ is a pushout of $\tilde{g}_1$, it suffices to apply Lemma~\ref{lem:lotsOfWork} to show that $\tilde{g}_1\boxprod (\Lambda^2_i \subset \Delta^2)\in\mathcal{M}$.

    \item $g_1\boxprod (\Lambda^n_i \subset \Delta^n)$, $n\geq 3$: We can apply Lemma~\ref{lem:joyal} to $(\Lambda^3_i\subset\Delta^3)\boxprod \tilde{g}_1$.

    \item $g_1\boxprod (\partial\Delta^n\subset\Delta^n)$, $n\geq 3$: Since $g_1\in \overline{\mathrm{Cell}}_{\geq 1}$, Lemma~\ref{lem:hope} implies $g_1\boxprod (\partial\Delta^n\subset\Delta^n) \in \mathcal{M}$.
\end{itemize}

\smallskip
\emph{Case $b_k\colon \partial\Delta^k\subset\Delta^k$, $k\geq 2$:} This case follows directly from Lemmas~\ref{lem:joyal} and~\ref{lem:hope}.
\end{proof}

Our next aim is to show that $[N(E),N(F)]$ has the structure of $\mathrm{FA}$ for each pair $E,F\in\PEA$. Thanks to Corollary~\ref{cor:KanFib}, it suffices consider the case $[\underline{1},N(F)]$, and then lift the result to $[N(E),N(F)]$.

\begin{proposition}\label{prop:smallFun}
    Let $\mathcal{L}_1$ and $\mathcal{L}_2$ be weakly saturated classes, where 
    \begin{itemize}
        \item    $\mathcal{L}_1$ is generated by (i--iii) from Theorem~\ref{thm:char} and $\mathrm{Horn}_{\geq 3}$,
        \item  $\mathcal{L}_2$ is generated by $\mathcal{L}_1$ and braiding~\eqref{diag:braiding}.
    \end{itemize}
    We have 
\begin{equation}
    \mathcal{L}_i\boxprod(\emptyset\subset\Sim^1_{\mathrm{red}})\subseteq\mathcal{L}_i.
\end{equation}
\end{proposition}
\begin{proof}
    First assume $f\colon A\rightarrow B$ is a generating morphism of $\mathcal{L}_0$, $\mathcal{L}_1$, respectivelly, not involving $\epsilon$-edges (associativity, $3$-coskeletality, cancelativity).
    In this case $B\times \Sim^1_{\mathrm{red}}$ does not contain any $\epsilon$-edge, hence $f\boxprod (\emptyset\subset\Sim^1_{\mathrm{red}})$ equals $f\boxprod (\emptyset\subset\Delta^1_{\mathrm{red}})$.
    
    Inclusion $\emptyset\hookrightarrow (\Delta^1_{\mathrm{red}})$ admits a splitting 
    \begin{equation}
        \emptyset \subset \Delta^0 \subset \Delta^1_{\mathrm{red}},
    \end{equation}
    where $\Delta^0\subset \Delta^1_{\mathrm{red}}$ is a pushout of $\partial \Delta^1\subset\Delta^1$ (Lemma~\ref{lem:reduction}). By Lemma~\ref{lem:boxAndClosure}, it is enough to show that
    $f\boxprod (\emptyset\subset \Delta^0)$ and $f\boxprod(\partial\Delta^1\subset\Delta^1)$ belong to $\mathcal{L}_1$, $\mathcal{L}_2$, respectively. Since $f$ does not contain any $\epsilon$-edge, the former morphism is isomorphic to $f$, hence it belongs to $\mathcal{L}_1$, $\mathcal{L}_2$, respectivelly.
We will solve the second case separately for each type of morphism $f$.

Associativity: we need to verify that
\begin{equation}
    (\partial_0\Delta^3\cup\partial_2\Delta^3)\times \Delta^1\coprod_{(\partial_0\Delta^3\cup\partial_2)\times\partial\Delta^1}  \Delta^3\times\partial\Delta^1 \quad\subset\quad \Delta^3\times\Delta^1\quad\in\mathcal{L}_1.
\end{equation}
Consider the adjancency graph of $4$-simplices of $\Delta^3\times\Delta^1$:

\begin{equation}\label{diag:graph1}
    \begin{tikzcd}[row sep=15pt, column sep=7pt]
\Delta^{\{0,\markindex{1},2,3,\markindex{3'}\}}\arrow[r,no head]&\Delta^{\{0,\markindex{1},2,2',\markindex{3'}\}}\arrow[r,no head]&\Delta^{\{0,1,1',2',\markindex{3'}\}}\arrow[r,no head]&\Delta^{\{\markindex{0},0',\markindex{1'},2',\markindex{3'}\}}\\
\end{tikzcd}
\end{equation}
It is straightforward to show that using Lemma~\ref{lem:essentialHelp} we can fill all the $4$-simplices.

$\mathrm{Horn}_{\geq 3}$: 
For each $n\geq 3$, the morphism $(\Lambda^n_i\subset \Delta^n)\boxprod (\partial_1\subset\Delta^1)$ belongs to $\mathcal{L}_1$ by Lemma~\ref{lem:joyal}.

$2$-coskeletality: For each $n\geq 3$ we have $(\partial\Delta^n\subset \Delta^n)\boxprod(\partial\Delta^1\subset\Delta^1)\in\mathcal{L}_1$ by Lemma~\ref{lem:hope}. 

Braiding: let $f$ be one of the morphism in~\eqref{diag:braiding}. Clearly, $f$ is a pushout of $\Lambda^2_0$, or $\Lambda^2_2$. 
Hence, it is enough to show that $(\Lambda^2_i\subset\Delta^2)\boxprod(\partial\Delta^1\subset\Delta^1)$, $i=0,2$, belongs to $\mathcal{L}_2$, which is easy.

It remains to deal with $\epsilon$-horn inclusions. Since cancelation property, we only need to consider $\ELambda^n_{i}\subset\Sim^n$, for $n=1,2$, $i=0,n$.
The cases $i=0,n$ are analogous, so we only focus on the case $i=0$.

$n=1$:  $(\ELambda^1_0\subset \Sim^1)\boxprod (\emptyset\subset (\Sim^1_{\mathrm{red}}))$ is pictured in the diagram~\eqref{diag:1hornExp}, where $\Delta^{\{0,1\}}$ and $\Delta^{\{0',1'\}}$ denote the same edge of the codomain cylinder.
\begin{equation}\label{diag:1hornExp}
    \begin{tikzcd}[row sep=6pt, column sep=20pt]
       0'&&&0'\arrow[r]&1'\\
       &{\ }\arrow[r,hook]&{\ }&&\\
       0\arrow[uu]&&&0\arrow[r]\arrow[uu]\arrow[ruu,Tarrow]&1\arrow[uu]
    \end{tikzcd}
    \end{equation}
  
    One can fill the cylinder in three steps. First we obtain the $\epsilon$-edge $\Sim^{\{0,1'\}}$, then we obtain the simplex $\Sim^{\{0,0'1'\}}$, and finally we yield $\Sim^{\{0,1,1'\}}$.

   $n=2$: the morphism $\ELambda^2_0\times \Sim^1_{\mathrm{red}} \subset \Sim^2\times\Sim^1_{\mathrm{red}}$ in concern is pictured in diagram~\eqref{diag:anotherPrism}, 
   where we identify the $2$-simplices $\Delta^{\{0,1,2\}}$ and $\Delta^{\{0',1',2'\}}$.

    \begin{equation}\label{diag:anotherPrism}
    \begin{tikzcd}[row sep=8pt, column sep=20pt]
        &2'&&&&&   &2'&\\
        0'\arrow[ru]\arrow[rr]&&1' &&&& 0'\arrow[ru]\arrow[rr]&&1'\arrow[lu]\\
        && &{}\arrow[rr,hook]&  &{}&&&\\
        &2\arrow[uuu]& && &&&2\arrow[uuu]&\\
        0\arrow[rruuu]\arrow[ruuuu,Tarrow]\arrow[uuu]\arrow[ur]\arrow[rr]&&1\arrow[uuu] &&&& 0\arrow[rruuu]\arrow[ruuuu,Tarrow]\arrow[uuu]\arrow[ur]\arrow[rr]&&1\arrow[lu]\arrow[luuuu]\arrow[uuu]
    \end{tikzcd}
\end{equation}
We show that the inclusion~\eqref{diag:anotherPrism} belongs to the weakly saturated class generated by all $\epsilon$-horns. As the first step we fill $\Sim^{\{0,1',2'\}}$, and then $\Sim^{\{0,0',1',2'\}}$. Next, we fill $\Sim^{\{0,1,2'\}}$ and $\Sim^{\{0,1,2,2'\}}$ (recall $\Delta^{\{0,1,2\}}=\Delta^{\{0',1',2'\}}$). 
It remains to fill the last $\epsilon$-horn to obtain $\Sim^{\{0,1,1',2'\}}$. This finishes the whole proof.
\end{proof}
\begin{corollary}
    Let $E, F\in\PEA$. Then there exists cancellative $G\in\RelFA$ such that $N(G)$ is isomorphic to the $\epsilon$-simplicial set $X=[N(E),N(F)]$. 
    Moreover, the vertices of $X$ are in bijection with the partial monoid morphisms $E\rightarrow F$ and there is an edge between two vertices if and only if the corresponding morphisms are conjugate.
 \end{corollary}
 \begin{proof}
    The interpretation of vertices and edges of $[N(E),N(F)]$ is clear (see Section~\ref{sec:5}). Let $\mathcal{L}_2$ be the class of morphisms in $\eSSet$ from Proposition~\ref{prop:bigFun}. We need to show that $\mathcal{L}_2\boxslash [N(E),N(F)]$.
    
    Let $f\colon A\hookrightarrow B$ be one of the generating morphism of  $\mathcal{L}_2$ and consider the following diagram:
    \begin{equation}\label{diag:lastLift}
        \begin{tikzcd}[row sep=15pt, column sep=15pt]
            A\arrow[dd,"f"]\arrow[r]&{[N(E),N(F)]}\arrow[d,"p"]\\
            &{[\underline{1},N(F)]}\arrow[d]\\
            B\arrow[r]&\underline{0}\\
        \end{tikzcd}
    \end{equation}
By Proposition~\ref{prop:smallFun}, we obtain a unique lift $l'\colon B\rightarrow [N(\underline{1}),N(F)]$. 

One can go through all the generators of $\mathcal{L}_2$ and easily check that the morphism $f$ belongs to the weakly saturated class generated by $\mathrm{Horn}$, $\mathrm{Cell}_{\geq 2}$, and $\Delta^1\subset\Sim^1$. Hence, Corollary~\ref{cor:KanFib} together with Theorem~\ref{thm:char} imply the desired property.
 \end{proof}
\begin{corollary}
    There is a $\RelFA$-enriched category $\PEA$, whose objects are pseudo effect algebras and the underlying category $\PEA_0$ is isomorphic to the ordinary category of pseudo effect algebras. 
\end{corollary}
\begin{proof}
    The proof is analogous to that of Theorem~\ref{thm:enrichedEA}.
\end{proof}
\section*{Conclusion}

In this work we have shown that $\epsilon$-simplicial sets provide a convenient framework for studying (pseudo) effect algebras. 
Among our main contributions, we highlight a stability result: if $N(F)$ satisfies certain lifting properties relevant to 
pseudo effect algebras, then the mapping space $[N(E),N(F)]$ inherits the same lifting properties, provided suitable conditions 
on $N(E)$ are met. Such stability phenomena are fundamental in the theory of quasi-categories and quasi-groupoids, and our results 
suggest that various 
theorems from the theory of quasi-categories and quasi-groupoids admit analogues in the context of Frobenius algebras.  

In future work, we plan to investigate mapping spaces $[A,B]$ associated with broader classes of Frobenius algebras in $\Rel$. 
To give a concrete question, consider the following situation: let $\mathcal{L}$ be a weakly saturated class of morphisms in $\SSet$ 
generated by inner horn inclusions (trivial cofibrations in the Joyal model structure for quasi-categories), 
and consider $f\colon A\to B\in \mathcal{L}$ and $g\colon C\to D\in \mathcal{L}^\boxslash$. The induced morphism
\[
p\colon [B,C]\;\longrightarrow\; [A,C]\times_{[A,D]}[B,D]
\]
has the interpretation that the codomain parametrizes lifting problems involving $(f,g)$, while the domain parametrizes their solutions.
 The fact that each lifting problem admits a solution, i.e.\ our assumption $f \boxslash g$, is internalised in Joyal's result (see \cite{L}, Cor.~2.3.2.5) that
 $p$ is a trivial Kan fibration. An interesting question is which lifting properties are satisfied by $p$ when $\mathcal{L}$ is a weakly
 saturated class generated by cofibrations encoding various classes of Frobenius algebras.

 \end{document}